\documentclass[11pt,reqno]{amsart}

\usepackage{amsmath,amssymb,amsthm}
\usepackage{bm}
\usepackage{natbib}
\usepackage{graphicx}
\usepackage{multirow}
\usepackage{here}
\usepackage{fullpage}
\usepackage{bm}
\usepackage{url}
\usepackage{color}
\usepackage{mathrsfs}

\makeatletter

\@addtoreset{equation}{section}
\makeatletter

\newtheorem{theorem}{Theorem}[section]
\newtheorem{corollary}{Corollary}[section]
\newtheorem{lemma}{Lemma}[section]
\newtheorem{proposition}{Proposition}[section]
\newtheorem{assumption}{Assumption}[section]
\theoremstyle{definition}
\newtheorem{definition}{Definition}[section]
\newtheorem{remark}{Remark}[section]

\renewcommand{\tilde}{\widetilde}
\renewcommand{\hat}{\widehat}

\DeclareMathOperator*{\argmin}{arg\,min}

\begin{document}

\title[]{On the estimation of locally stationary functional time series}
%\thanks{Corresponding author: Daisuke Kurisu.}
\thanks{D. Kurisu is partially supported by JSPS KAKENHI Grant Number 20K13468. The author would like to thank Alessia Caponera, Konstantinos Fokianos, Kengo Kato, Muneya Matsui, and Taisuke Otsu for their helpful comments and suggestions.
} 

\author[D. Kurisu]{Daisuke Kurisu}

\date{First version: April 28, 2021. This version: \today}

\address[D. Kurisu]{Center for Spatial Information Science, The University of Tokyo\
5-1-5, Kashiwanoha, Kashiwa-shi, Chiba 277-8568, Japan
}
\email{daisukekurisu@csis.u-tokyo.ac.jp}
\subjclass[2010]{60F05, 62G05, 62G20}

\begin{abstract}
This study develops an asymptotic theory for estimating the time-varying characteristics of locally stationary functional time series (LSFTS). We investigate a kernel-based method to estimate the time-varying covariance operator and the time-varying mean function of an LSFTS. In particular, we derive the convergence rate of the kernel estimator of the covariance operator and associated eigenvalue and eigenfunctions and establish a central limit theorem for the kernel-based locally weighted sample mean. As applications of our results, we discuss methods for testing the equality of time-varying mean functions in two functional samples. 

\medskip

\noindent
\textit{Keywords}: functional time series, locally stationary process, principal component analysis, nonparametric estimation
\end{abstract}

%\keywords{functional time series, locally stationary process, principal component analysis}

\maketitle

\section{Introduction}

Functional time series analysis has been a growing interest in many scientific fields, such as environmetrics (\cite{AuDuNo15}), biometrics (\cite{ChMu09}), demographics (\cite{LiRoSh20}), and finance (\cite{KoZh12}, \cite{ChLeTu16}). We refer to \cite{RaSi05}, \cite{FeVi06}, and \cite{HoKo12} as the standard references on functional data and functional time series analysis. 

In the literature of functional time series analysis, most of the studies are based on stationary models (e.g., \cite{Bo00, Bo02}, \cite{DeSh05}, \cite{AnPaSa06}, \cite{HoKoRe13}, \cite{AuDuNo15}, and \cite{DeKoVo20}). However, many functional time series exhibit nonstationary behavior. Typical examples include the daily patterns of temperature observed in a region and daily curves of implied volatility of an option as a function of moneyness. We can also find other examples of nonstationary functional time series in \cite{vaEi18}. One way to model nonstationary behavior is provided by the theory of locally stationary processes.

A locally stationary process, as proposed by \cite{Da97}, is a nonstationary time series %$\{X_{t,T}\}_{t=1}^{T}$ 
that allows its parameters to be time-varying and can be approximated by a stationary time series % $\{X_{t}^{(u)}\}_{t \in \mathbb{Z}}$ when $t/T$ is close to 
at each rescaled time point. %$u \in [0,1]$.
This property enables us to develop asymptotic theories for the estimation of time-varying characteristics. The rigorous definition of local stationarity for functional data is given in Section \ref{Sec: setting}. There is a vast literature on locally stationary (multivariate) time series. We mention \cite{DaSu06}, \cite{FrSaSu08}, \cite{Kr09}, \cite{KoLi12}, \cite{Vo12} \cite{Zh14}, \cite{ZhWu15}, \cite{Tr17, Tr19}, and \cite{KuFuKo22}, to name a few. Since the introduction of the notion of locally stationary processes, it has been extended in several directions, such as spatial data (\cite{Pe18}, \cite{Ku21}) and spatio-temporal data (\cite{MaYa18}). As recent important contributions in the literature of functional time series analysis, we refer to \cite{vaEi18}, \cite{Auva20}, \cite{BuDeHe20}, and \cite{vaDe21}. \cite{vaEi18} extend the notion of locally stationary processes to functional time series that take values in a Hilbert space. The authors investigate frequency domain methods for locally stationary functional data. \cite{Auva20} develop a method for testing the null hypothesis that the observed functional time series is stationary against the hypothesis that it is locally stationary. \cite{BuDeHe20} investigate testing procedure for second-order stationarity of locally stationary functional time series. \cite{vaDe21} develop a similarity measure based on time-varying spectral density operators for cluster analysis of nonstationary functional time series. We also mention \cite{Ku21b}, who investigates nonparametric regression for locally stationary functional time series and \cite{DeWu21} for the construction of confidence surfaces of the mean function of locally stationary functional time series.

The objective of this paper is to develop an asymptotic theory for a kernel-based method to estimate the time-varying covariance operator and the time-varying mean function of the locally stationary functional time series at each rescaled time point. 

The first objective is related to the problem of estimating time-varying eigenvalues and eigenfunctions of the covariance operator, which is important to develop methods for functional principal component analysis (FPCA) that are based on those time-varying (or local) characteristics. It has been recognized that FPCA is a dimension reduction technique in many scientific research fields, such as chemical engineering, functional magnetic resonance imaging, and environmental science. We refer to \cite{RaSi05} for a detailed comparison between PCA in multivariate setting and FPCA. Specifically, we consider a kernel-based estimator of the time-varying covariance operator and derive its convergence rate (i.e. Theorem \ref{Thm1}). The result also enables us to derive convergence rates of estimators of associated eigenvalues and eigenfunctions. 

The second objective is to extend the statistical methods based on the sample mean of functional time series (e.g., testing the equality of mean functions in two functional samples (cf. \cite{HoKoRe13} and \cite{DeKoVo20})) developed under the assumption of stationarity to nonastationary cases. We consider a kernel-based locally weighted sample mean and show that it converges to a Gaussian process (i.e. Theorem \ref{Thm2}). We also propose a kernel-based estimator of the long-run covariance kernel function of the Gaussian process and establish its consistency. 

As possible applications of our results, we discuss problems on 
estimating local mean functions and testing the equality of local mean functions of locally stationary functional time series. For the implementation of the two sample test, we use FPCA-based projection methods to retain most of the information carried by the original process and provide a method for selecting the number of principal components. Consequently, we extend well-known results developed for stationary functional time series to our framework. 

The organization of this paper is as follows. In Section \ref{Sec: setting}, we introduce the notion of local stationarity for functional time series that take values in a Hilbert space and dependence structure of the functional time series. In Section \ref{Sec: Main}, we present the main results. In Section \ref{Sec: Applications}, we discuss some applications of our results. All proofs are included in the Appendix. 

\subsection{Notations}
For any positive sequences $a_{n}$ and $b_{n}$, we write $a_{n} \lesssim b_{n}$ if a constant $C >0$ independent of $n$ exists such that $a_{n} \leq Cb_{n}$ for all $n$ and  $a_{n} \sim b_{n}$ if $a_{n} \lesssim b_{n}$ and $b_{n} \lesssim a_{n}$. %, and $a_{n} \ll b_{n}$ if $a_{n}/b_{n} \to 0$ as $n \to \infty$. 
For any $a,b \in \mathbb{R}$, we write $a \vee b=\max\{a,b\}$ and $a \wedge b=\min\{a,b\}$. We use the notations $\stackrel{d}{\to}$ and $\stackrel{p}{\to}$ to denote convergence in distribution and convergence in probability, respectively. For a Hilbert space with inner product $\langle \cdot, \cdot \rangle$, let $\|x\| = \sqrt{\langle x, x \rangle}$, $x \in H$ and let $\mathcal{L}$ denote the space of bounded (continuous) linear operators on $H$ with the norm 
\[
\|\Psi\|_{\mathcal{L}} = \sup\{\|\Psi(x)\|: \|x\|\leq 1, x \in H\}. 
\]
An operator $\Psi \in \mathcal{L}$ is said to be compact if there exist two orthonormal basis $\{v_{j}\}$ and $\{f_{j}\}$, and a real sequence $\{\lambda_{j}\}$ converging to zero such that
\begin{align}\label{def: compact-op-mainbody}
\Psi(x) = \sum_{j=1}^{\infty}\lambda_{j}\langle x, v_{j} \rangle f_{j},\ x \in H. 
\end{align}
Further, a compact operator with the singular value decomposition (\ref{def: compact-op-mainbody}) is said to be a Hilbert-Schmidt operator if $\sum_{j=1}^{\infty}\lambda_{j}^{2}<\infty$. Let $\mathcal{S}$ be the space of Hilbert-Schmidt operators on $H$ with inner product
\[
\langle \Psi_{1}, \Psi_{2}\rangle_{\mathcal{S}} = \sum_{j=1}^{\infty}\langle \Psi_{1}(e_{j}), \Psi_{2}(e_{j})\rangle,
\]
where $\{e_{j}\}$ is an arbitrary orthonormal basis of $H$. The Hilbert-Schmidt norm on $\mathcal{S}$ is defined as $\|\Psi\|_{\mathcal{S}}=\sqrt{\sum_{j=1}^{\infty}\lambda_{j}^{2}}$, which does not depend on the choice of the orthonormal basis. Note that $(\mathcal{L}, \|\cdot \|_{\mathcal{L}})$ is a Banach space and $(\mathcal{S}, \|\cdot\|_{\mathcal{S}})$ is a separable Hilbert space. We refer to \cite{Bo00} for details on the properties of operators on a Hilbert space. For $x \in H$, define the tensor product $f \otimes g: H \otimes H \to H$ as the bounded linear operator $(f \otimes g)(x)=\langle g, x \rangle f$.

\section{Settings}\label{Sec: setting}

In this section, we introduce the notion of a locally stationary functional time series that extends the notion of a locally stationary process introduced by \cite{Da97}. Furthermore, we discuss dependence structures of the functional time series. For some probability space $(\Omega, \mathcal{A},P)$ and for $p \geq 1$, let $L^{p}(\Omega, \mathcal{A},P)$ denote the space of real-valued random variables such that $\|X\|_{p}=(E[|X|^{p}])^{1/p}<\infty$. Let $L^{2} = L^{2}([0,1]^d)$ denote a Hilbert space with  inner product $\langle \cdot, \cdot \rangle$ defined by
\[
\langle x, y \rangle = \int_{[0,1]^d} x(t)y(t)dt,\ x, y \in H=L^{2}.
\]
Further, let $L^{p}_{H}=L^{p}_{H}(\Omega, \mathcal{A},P)$ denote the space of $H$-valued random functions $X$ such that
\[
v_{p}(X) = (E[\|X\|^{p}])^{1/p} = \left(E\left[\left(\int_{[0,1]^d} X^{2}(t)dt\right)^{p/2}\right]\right)^{1/p}<\infty. 
\]  

\subsection{Local stationarity}
Intuitively, a functional time series $\{X_{t,T}\}_{t \in \mathbb{Z}}$ in $L^{p}_{H}$ is locally stationary if it behaves approximately stationary in local time. We refer to \cite{DaSu06} and \cite{DaRiWu19} for the idea of locally stationary time series and its general theory, and \cite{vaEi18} and \cite{Auva20} for the notion of local stationarity for a Hilbert space-valued time series. To ensure that it is locally stationary around each rescaled time point $u \in [0,1]$, a process $\{X_{t,T}\}$ in $L^{p}_{H}$ can be approximated by a stationary functional time series $\{X_{t}^{(u)}\}$ in $L^{p}_{H}$ stochastically. This concept can be defined as follows. 

\begin{definition}\label{def: LSfTS}
A sequence of $H$-valued stochastic process $\{X_{t,T}\}_{t=1}^{T}$ in $L^{p}_{H}$ is locally stationary if, for each rescaled time point $u \in [0,1]$, there exists an associated $H$-valued process $\{X_{t}^{(u)}\}_{t \in \mathbb{Z}}$ in $L^{p}_{H}$ with the following properties: 
\begin{itemize}
\item[(i)] $\{X_{t}^{(u)}\}_{t \in \mathbb{Z}}$ is strictly stationary.
\item[(ii)] It holds that 
\begin{align}\label{def: LSfTS-ineq}
\|X_{t,T}-X_{t}^{(u)}\| \leq \left(\left|{t \over T} - u\right| + {1 \over T}\right)U_{t,T}^{(u)}\ a.s.,
\end{align}
for all $1 \leq t \leq T$ where $\{U_{t,T}^{(u)}\}$ is a process of positive variables satisfying $E[(U_{t,T}^{(u)})^{\rho}]<C$ for some $\rho>0$, $C<\infty$ independent of $u, t$, and $T$.
\end{itemize}
\end{definition}

Definition \ref{def: LSfTS} is a natural extension of the notion of local stationarity for real-valued time series introduced in \cite{Da97}.  %Moreover, our definition corresponds to that of \cite{vaEi18} (Definition 2.1) when $p=2$. The authors also give sufficient conditions that an $H$-valued stochastic process $\{X_{t,T}\}$ satisfies (\ref{def: LSfTS-ineq}) with $\rho=2$. 

\subsection{Dependence structure}

For each $u \in [0,1]$, we assume that $\{X_{t}^{(u)}\}$ is $L^{p}$-$m$-approximable, that is, $X_{t}^{(u)} \in L^{p}_{H}$ is of the form
\[
X_{t}^{(u)}=f_{u}(\varepsilon_{t}, \varepsilon_{t-1},\varepsilon_{t-2},\dots),
\]
where $\varepsilon_{j}$ are i.i.d. elements taking values in a measurable space $S$, and $f_{u}$ is a measurable function $f_{u}: S^{\infty} \to H$. Note that $X_{t}^{(u)}$ is strictly stationary. We also assume that if $\{\varepsilon_{j,k}^{(m)}\}$ is an independent copy of $\{\varepsilon_{j}\}$ defined on the same probability space, then letting
\begin{align}\label{m-approx-ver}
X_{m,t}^{(u)} &= f_{u}(\varepsilon_{t}, \varepsilon_{t-1},\dots, \varepsilon_{t-m+1}, \varepsilon_{t,t-m}^{(m)}, \varepsilon_{t,t-m-1}^{(m)},\dots),
\end{align}
we have 
\[
\sum_{m \geq 1}v_{p}(X_{t}^{(u)}-X_{m,t}^{(u)})<\infty.
\]

We can show that a wide class of functional time series (e.g., functional AR(1) process and functional ARCH(1) process) is $L^{p}$-$m$-approximable under some regularity conditions. Hence the concept of $L^p$-$m$ approximability of functional time series is popular in the literature of functional time series analysis. See \cite{HoKo10} and \cite{HoKo12} for details on the properties of $L^{p}$-$m$-approximable random functions. We also refer to \cite{ZhDe20} and \cite{DeWu21} who consider physical representation of functional time series as a different dependence structure (see, for example, Section 2 in \cite{ZhDe20} for the connection and difference between $L^p$-$m$ approximability and physical dependence).

\section{Main results}\label{Sec: Main}

In this section, we introduce a kernel-based method to estimate the time-varying covariance operator and the time-varying mean function of a locally stationary time series at each rescaled time point $u \in [0,1]$. 

\subsection{Estimation of eigenvalues and eigenfunctions of local covariance operators}

We make the following assumption on the process $\{X_{t,T}\}$. 
\begin{assumption}\label{Ass-M}%[Assumption M] 
The process $\{X_{t,T}\}$ in $L^{4}_{H}$ is locally stationary. In particular, there exists an $L^{4}$-$m$-approximable process $\{X_{t}^{(u)}\}$ in $L^{4}_{H}$ that satisfies Definition \ref{def: LSfTS} with $p=4$ and some $\rho \geq 2$. We also assume that $E[X_{t}^{(u)}] =0$ in $L^{2}$.
\end{assumption}

\noindent
\textbf{Example.} We give an example of locally stationary functional time series that satisfies Assumption \ref{Ass-M}. Let $\{H_k(u;\mathcal{F}_{t}^{(k)})\}_{t \in \mathbb{Z}}$ be stationary time series for each $k \in \mathbb{N}$ and $u \in [0,1]$ that satisfy the following conditions: 
\begin{itemize}
\item[(E1)] $H_{k}(u;\mathcal{F}_{t}^{(k)}) = G_k(u; \zeta_{t}^{(k)}, \zeta_{t-1}^{(k)},\zeta_{t-2}^{(k)},\dots)$ where $\zeta_{t}^{(k)}$ are i.i.d. elements taking values in a measurable space $S$, and $G_k(u; \cdot)$ is a measurable function $G_k(u;\cdot): S^{\infty} \to \mathbb{R}$.
\item[(E2)] $\zeta^{(k)}=\{\zeta_{t}^{(k)}\}_{t \in \mathbb{Z}}$ are mutually independent.
\item[(E3)] There exist random variables $\bar{H}_{k}(\mathcal{F}_{t}^{(k)})$ such that
\begin{align*}
&|H_k(t/T; \mathcal{F}_{t}^{(k)}) - H_k(u;\mathcal{F}_{t}^{(k)})| \leq \left|{t \over T} - u\right|\bar{H}_k(\mathcal{F}_{t}^{(k)}),\\
&\sum_{k \geq 1}E\left[\left|H_{k}(u;\mathcal{F}_{t}^{(k)})\right|^4\right]^{1/2}<\infty,\ \sum_{k \geq 1}E\left[\left|\bar{H}_k(\mathcal{F}_{t}^{(k)})\right|^2\right] <\infty. 
\end{align*}
\item[(E4)] Let $\{\zeta_{t,j}^{(k,m)}\}$ be an independent copy of $\{\zeta_{t}^{(k)}\}$. Assume that
\begin{align*}
\sum_{m \geq 1}\left(\sum_{k \geq 1}E\left[\left(H_k(u;\mathcal{F}_{t}^{(k)}) - H_{k}(u;\mathcal{F}_{t}^{(k,m)})\right)^4\right]^{1/2}\right)^{1/2}<\infty,
\end{align*}
where $H_{k}(u;\mathcal{F}_{t}^{(k,m)}) = G_k(u;\zeta_{t}^{(k)}, \zeta_{t-1}^{(k)}, \cdots, \zeta_{t-m+1}^{(k)}, \zeta_{t,t-m}^{(k,m)}, \zeta_{t,t-m-1}^{(k,m)}, \dots)$. 
\end{itemize}
For a given orthogonal basis $\{\phi_k(s)\}_{k = 1}^{\infty}$ of $L^2([0,1]^d)$, consider the functional time series defined by 
\begin{align*}
X_{t}^{(u)}(s) &= \sum_{k \geq 1}H_k(u;\mathcal{F}_{t}^{(k)})\phi_k(s),\ X_{m,t}^{(u)}(s) = \sum_{k \geq 1}H_k(u;\mathcal{F}_{t}^{(k,m)})\phi_k(s),\\
X_{t,T}(s) &= X_{t}^{(t/T)}(s),\  1 \leq t \leq T.
\end{align*}

\begin{proposition}\label{Prop: example}
Under Conditions (E1)-(E4), the processes $X_{t,T}$ and $X_{t}^{(u)}$ satisfy Assumption \ref{Ass-M}. 
\end{proposition} 

Now we move on to our main results. Let $C_{u}: H \to H$ be the covariance operator of a zero-mean stationary process $\{X_{t}^{(u)}\}$, that is, $C_{u}(\cdot)=E[X_{t}^{(u)} \otimes X_{t}^{(u)}]$. As in the multivariate time series case, $C_{u}$ admits the singular value decomposition
\[
C_{u}(x) = \sum_{j=1}^{\infty}\lambda_{u, j}\langle v_{u,j}, x \rangle v_{u, j},
\]
where $\{\lambda_{u,j}\}$ are the non-negative eigenvalues (in descending order) and $\{v_{u,j}\}$ are the corresponding normalized eigenfunctions, that is, $C_{u}(v_{u,j}) = \lambda_{u,j}v_{u,j}$ and $\|v_{u,j}\| =1$. The $\{v_{u,j}\}$ form an orthonormal
basis of $H$. Hence $X_{t}^{(u)}$ allows for the Karhunen-Lo\'eve representation $X_{t}^{(u)} =\sum_{j=1}^{\infty}\langle X_{t}^{(u)}, v_{u,j} \rangle v_{u,j}$. The coefficients $\langle X_{t}^{(u)}, v_{u,j} \rangle$ in this expansion are called the FPC scores of $X_{t}^{(u)}$.

Consider the empirical local covariance operator 
\begin{align*}
\hat{C}_{u} = {1 \over Th}\sum_{t=1}^{T}K_{1,h}\left(u - {t \over T}\right)X_{t,T} \otimes X_{t,T},
\end{align*}
where $K_{1}$ is a one-dimensional kernel function and we use the notation $K_{1,h}(x)=K_{1}(x/h)$, and $h$ is a sequence of positive constants (bandwidths) such that $h=h_{T} \to 0$ as $T \to \infty$. When the mean function of $X_{t}^{(u)}$ is unknown, that is, $E[X_{t}^{(u)}]= m(u, \cdot)$, then we can estimate $m(u, \cdot)$ by a kernel-based locally weighted sample mean $\bar{X}_{T}^{(u)}={1 \over Th}\sum_{t=1}^{T}K_{1,h}(u-t/T)X_{t,T}$. In this case, $\hat{C}_{u}$ is given by 
\[
\hat{C}_{u} = {1 \over Th}\sum_{t=1}^{T}K_{1,h}\left(u - {t \over T}\right)\left(X_{t,T}-\bar{X}_{T}^{(u)}\right) \otimes \left(X_{t,T}-\bar{X}_{T}^{(u)}\right).
\]
Under some regularity conditions, we can show $E\left[\|\bar{X}_{T}^{(u)}-m(u, \cdot)\|^{2}\right] = O\left(h^{2} + {1 \over Th} + {1 \over T^{2}h^{4}}\right)$. We discuss the asymptotic properties of $\bar{X}_{T}^{(u)}$ in Section \ref{subsec: weighted mean}. See Theorem \ref{Thm2} below and its proof for details. 
We make the following assumption on the kernel function $K_{1}$. 
\begin{assumption}\label{Ass-KB}%[Assumption KB]
The kernel $K_{1}$ is symmetric around zero, bounded, and has a compact support, i.e., $K_{1}(v) = 0$ for all $|v|>C_{1}$ for some $C_{1}<\infty$. Moreover, $\int K_{1}(z)dz=1$ and $K_{1}$ is Lipschitz continuous, i.e., $|K_{1}(v_{1}) - K_{1}(v_{2})| \leq C_{2}|v_{1} - v_{2}|$ for some $C_{2}<\infty$ and all $v_{1},v_{2} \in \mathbb{R}$. 
\end{assumption}

The next result provides the convergence rate of the kernel estimator $\hat{C}_{u}$ at each rescaled time point $u$. 
\begin{theorem}\label{Thm1}
Under Assumptions \ref{Ass-M} and \ref{Ass-KB}, we have
\begin{align}\label{error-bound-cov-op-Thm1}
\|\hat{C}_{u} - C_{u}\|_{\mathcal{L}} &= O_{P}\left(h + \sqrt{{1 \over Th}} + {1 \over Th^{2}}\right)
\end{align}
for $u \in [C_{1}h, 1-C_{1}h]$.
\end{theorem}
The first term $h$ in the error bound (\ref{error-bound-cov-op-Thm1}) comes from the local stationarity of $X_{t,T}$, i.e., the approximation error of $X_{t,T}$ by $X_{t}^{(u)}$, and the second and the third terms correspond to variance and bias terms, respectively. 

Let $\{\hat{v}_{u,j}\}_{j=1}^{q}$ be the eigenfunctions of $\hat{C}_{u}$ corresponding to the $q$ largest eigenvalues and set $\hat{c}_{u,j}$ as 
\[
\hat{c}_{u,j}=\text{sign}(\langle \hat{v}_{u,j}, v_{u,j}\rangle).
\]  
Further, let $\{\hat{\lambda}_{u,j}\}$ be estimators of $\{\lambda_{u,j}\}$ that satisfy
\[
\hat{C}_{u}(\hat{v}_{u,j})=\hat{\lambda}_{u,j}\hat{v}_{u,j},\ 1\leq j \leq q.
\]
The next result immediately follows from Theorem \ref{Thm1}. It shows that $\{\hat{\lambda}_{u,j}\}_{j=1}^{q}$ are consistent estimators of $\{\lambda_{u,j}\}_{j=1}^{q}$, and $\{\hat{v}_{u,j}\}_{j=1}^{q}$ consistently estimates $\{v_{u,k}\}_{j=1}^{q}$ up to sign change.
\begin{corollary}\label{Cor1}
Suppose $\lambda_{u,1}>\lambda_{u,2}>\dots>\lambda_{u,q}>\lambda_{u,q+1}$. Under Assumptions \ref{Ass-M} and \ref{Ass-KB}, we have
\begin{align}
\max_{1 \leq j \leq q}|\hat{\lambda}_{u,j} - \lambda_{u,j}| &= O_{P}\left(h + \sqrt{{1 \over Th}} + {1 \over Th^{2}}\right), \label{rate1}\\ 
\max_{1 \leq j \leq q}\|\hat{c}_{u,j}\hat{v}_{u,j} - v_{u,j}\| &= O_{P}\left(h + \sqrt{{1 \over Th}} + {1 \over Th^{2}}\right) \label{rate2}
\end{align}
for $u \in [C_{1}h, 1-C_{1}h]$. 
\end{corollary}
The convergence rates (\ref{rate1}) and (\ref{rate2}) are optimized by choosing $h \sim T^{-1/3}$ and the optimized rates are 
\begin{align*}
|\hat{\lambda}_{u,j} - \lambda_{u,j}| &= O_{P}\left(T^{-1/3}\right),\ \|\hat{c}_{u,j}\hat{v}_{u,j} - v_{u,j}\| = O_{P}\left(T^{-1/3}\right). 
\end{align*}

\subsection{Asymptotic normality of locally weighted sample means}\label{subsec: weighted mean}

Now we make the following assumption instead of Assumption \ref{Ass-M}, 
\begin{assumption}\label{Ass-M2}%[Assumption M2] 
The process $\{X_{t,T}\}$ in $L^{2}_{H}$ is locally stationary. In particular, there exists an $L^{2}$-$m$-approximable process $\{X_{t}^{(u)}\}$ in $L^{2}_{H}$ that satisfies $E[X_{t}^{(u)}]=0$ in $L^{2}$ and satisfies Definition \ref{def: LSfTS} with $p=2$ and some $\rho \geq 2$.
\end{assumption}
Define 
\[
\bar{X}_{T}^{(u)} = {1 \over Th}\sum_{t=1}^{T}K_{1,h}\left(u-{t \over T}\right)X_{t,T}. 
\]
The next result is a central limit theorem of $\bar{X}_{T}^{(u)}$. 
\begin{theorem}\label{Thm2}
Suppose $Th \wedge  {1 \over Th^{3}} \to \infty$ as $T \to \infty$. Under Assumptions \ref{Ass-KB} and \ref{Ass-M2}, for each $u \in (0,1)$, we have
\begin{align*}
\sqrt{Th}\bar{X}_{T}^{(u)} &\stackrel{d}{\to} G^{(u)}\ \text{in $L^{2}$},
\end{align*}
where $G^{(u)}$ is a Gaussian element in $L^2$ with $E[G^{(u)}(s)]=0$ and with the covariance kernel function $E[G^{(u)}(s_{1})G^{(u)}(s_{2})] = c^{(u)}(s_{1},s_{2})\int K_{1}^{2}(z)dz$ where 
\begin{align*}
c^{(u)}(s_{1},s_{2}) &= E[X_{0}^{(u)}(s_{1})X_{0}^{(u)}(s_{2})] + \sum_{t \geq 1}E[X_{0}^{(u)}(s_{1})X_{t}^{(u)}(s_{2})] + \sum_{t \geq 1}E[X_{0}^{(u)}(s_{2})X_{t}^{(u)}(s_{1})]. 
\end{align*}
\end{theorem}
Note that the infinite sums in the definition of the kernel $c^{(u)}$ coverge in $L^{2}([0,1]^d \times [0,1]^d)$, that is, $\int_{[0,1]^d \times [0,1]^d} \left(c^{(u)}(s_{1},s_{2})\right)^{2}ds_{1}ds_{2}<\infty$. See Appendix \ref{Appendix: proof} for details. If $E[X_{t}^{(u)}] = m(u, \cdot) \neq 0$, the same result of Theorem \ref{Thm2} holds by replacing $\bar{X}_{T}^{(u)}$ and $X_{t}^{(u)}$ with $\bar{X}_{T}^{(u)}-m(u, \cdot)$ and $X_{t}^{(u)} - m(u, \cdot)$, respectively. In Section 4, we discuss applications of Theorem \ref{Thm2} and it plays an important role in deriving the limit distributions of statistics for testing the equality of time-varying mean functions of two functional time series.

For the inference on the local mean function $m(u, \cdot)$, we need to estimate the long-run local covariance kernel $c^{(u)}$. We propose the following estimator as an estimator of $c^{(u)}$. 
\begin{align}\label{def: cov-kernel-est}
\hat{c}^{(u)}(s_{1},s_{2}) &= \hat{\gamma}_{0}^{(u)}(s_{1},s_{2}) + \sum_{t=1}^{T-1}K_{2}\left({t \over b}\right)\left\{\hat{\gamma}_{t}^{(u)}(s_{1},s_{2}) + \hat{\gamma}_{t}^{(u)}(s_{2},s_{1})\right\}
\end{align}
where 
\begin{align*}
\hat{\gamma}_{t}^{(u)}(s_{1},s_{2}) &= {1 \over Th}\sum_{j=t+1}^{T}K_{1,h}\left(u-{j \over T}\right)\left(X_{j,T}(s_{1}) - \bar{X}_{T}^{(u)}(s_{1})\right)\left(X_{j-t,T}(s_{2}) - \bar{X}_{T}^{(u)}(s_{2})\right) 
\end{align*}
for $1 \leq t \leq T-1$, $b$ is a sequence of positive constants (bandwidths) such that $b \to \infty$ as $T \to \infty$, and $K_{2}$ is a kernel function that satisfies the following condition. 
\begin{assumption}\label{Ass-KB2}%[Assumption KB2]
The kernel $K_{2}$ is bounded, and has a compact support, i.e., $K_{2}(v) = 0$ for all $|v|>C_{2}$ for some $0<C_{2}<\infty$. Moreover, we assume that $K_{2}(0)=1$.
\end{assumption}

For the bandwidth $b$, we make the following assumption. 
\begin{assumption}\label{Ass-b} 
As $T \to \infty$, $\min\{Th, {1 \over Th^{3}}, b, {Th \over b}\} \to \infty$.
\end{assumption}

%(${T \over h} \gg h^{-2}$, $T\sqrt{Th} \gg Th$. if $Th^{3} \ll1$, then $h^{-2} \gg Th$. In this case, Assumption \ref{Ass-b} is reduced to $b + {Th \over b^{2}} \to \infty$.)

Under the above assumptions, we can establish the following result.

\begin{theorem}\label{Thm3}
Assume that Assumptions \ref{Ass-M} (with $\rho=4$), \ref{Ass-KB}, \ref{Ass-KB2}, and \ref{Ass-b} are satisfied. Additionally, assume that $\lim_{m \to \infty}mv_{2}(X_{0}^{(u)} - X_{m,0}^{(u)}) = 0$. Then for each $u \in (0,1)$, we have
\begin{align}
\int_{[0,1]^d} \int_{[0,1]^d} \left\{\hat{c}^{(u)}(s_{1},s_{2}) - c^{(u)}(s_{1},s_{2})\right\}^{2}ds_{1}ds_{2} \stackrel{p}{\to} 0\ \text{as $T \to \infty$},
\end{align}
where $c^{(u)}$ is the covariance kernel that appears in Theorem \ref{Thm2}. 
\end{theorem}

Theorem \ref{Thm3} enables us to construct local FPCA-based statistics for testing the equality of the local mean functions. See Section \ref{Sec: Applications} for details. 

\section{Applications}\label{Sec: Applications}

In this section, we discuss two possible applications of our results. The first is the estimation of the local mean function of the locally stationary functional time series. The second is a two sample test for the difference of local mean functions, which generalizes the results of \cite{HoKoRe13} for the strictly stationary case to our settings. 

\subsection{Estimation of local mean functions}
Let $D$ be a compact subset of $\mathbb{R}^{d}$ and let $H=L^{2}$, the space of real-valued square integrable functions on $D$. Consider the following model: 
\begin{align}\label{model: spatio-temporal}
X_{t,T}(s) &= m\left({t \over T},s \right) + v_{t,T}(s),\ s \in D
\end{align}
where $m(u,\cdot)$ is an element of $H$ and $\{v_{t,T}\}$ is a zero-mean $H$-valued locally stationary process that satisfies Assumption \ref{Ass-M2} with $\{X_{t,T}\}=\{v_{t,T}\}$ and $\{X_{t}^{(u)}\} = \{v_{t}^{(u)}\}$. Assume that $m(u,\cdot)$ is Lipschitz continuous with respect to the norm $\|\cdot\|$ as a function of $u$, i.e., there exists a positive constant $C_{m}$ such that $\|m(v,\cdot)-m(u,\cdot)\| \leq C_{m}|u-v|$ for $u,v \in [0,1]$. Let $X_{t}^{(u)}(s)=m(u,s)+v_{t}^{(u)}(s)$. Then for each $u \in [0,1]$, $\{X_{t}^{(u)}\}$ is a $L^{2}$-$m$-approximable random functions and 
\begin{align*}
\|X_{t,T}-X_{t}^{(u)}\| &\leq \|m(t/T,\cdot)-m(u,\cdot)\|  + \|v_{t,T} - v_{t}^{(u)}\| \leq (C_{m}+1)\left(\left|u-{t \over T}\right| + {1 \over T}\right)(1+U_{t,T}^{(u)}). 
\end{align*} 
Then $\{X_{t,T}\}$ is a locally stationary $H$-valued functional time series in $L^{2}_{H}$. 
\begin{remark}
When $d=1$, the model (\ref{model: spatio-temporal}), with constant mean function (i.e., $m(u,s) = m(s)$) and stationary error process (i.e. $\{v_{t,T}\}$ itself is stationary), includes many examples, such as financial data (\cite{LiRoSh20}), biology (\cite{ChMu09}) and environmental data (\cite{AuDuNo15}).
When $d=2$, the model (\ref{model: spatio-temporal}) can be seen as a spatio-temporal model with a time-varying mean function. The model is also called a surface time series.  Surface data provide an alternative approach to analyzing spatial data, where the continuous realization of a
random field is considered as a unit. This approach can have computational advantages, especially if the locations where data are observed are dense in space. For example, $\{X_{t,T}\}$ would be an observation of functional surfaces of the daily records of temperature or precipitation with a time-varying mean function. Other examples  include acoustic phonetic data (\cite{AsPiTa17}), data from two surface electroencephalograms (EEG) (\cite{CrCaLuZiPu11}), and proteomics data (\cite{MoBaHeSaGu11}). We also refer to \cite{MaGe20} as an overview of recent developments in functional data including surface time series.
\end{remark}

Let $\hat{m}(u,\cdot)=\bar{X}_{T}^{(u)}$. Applying Theorem \ref{Thm2}, we can estimate the time-varying mean function $m(u, \cdot)$:
\[
E\left[\|\hat{m}(u,\cdot)-m(u,\cdot)\|^{2}\right] = O\left({1 \over Th}\right). 
\]
If we do not assume ${1 \over Th^{3}} \to \infty$ as $T \to \infty$, then we can also show 
\begin{align}\label{conv-rate-mean}
E\left[\|\hat{m}(u,\cdot)-m(u,\cdot)\|^{2}\right] = O\left(h^{2} + {1 \over Th} + {1 \over T^{2}h^{4}}\right). 
\end{align}
The term $h^{2}$ in the convergence rate arises from the local stationarity of $X_{t,T}$, i.e., the approximation error of $X_{t,T}$ by a stationary process $X_{t}^{(u)}$. The second and the third terms correspond to variance and bias terms, respectively. The convergence rate (\ref{conv-rate-mean}) is optimized by choosing $h \sim T^{-1/3}$ and the optimized rate is $E\left[\|\hat{m}(u,\cdot)-m(u,\cdot)\|^{2}\right] = O(T^{-2/3})$. 

\subsection{Two sample problems}

Let $\{X_{t,T}\}$ be an $H$-valued locally stationary functional time series. In principle, we can test whether the underlying functional time series $\{X_{t,T}\}$ is stationary or not. \cite{Auva20} proposed a method for testing the null hypothesis that $\{X_{t,T}\}$ is stationary against the alternative hypothesis that the process is locally stationary. Further, \cite{HoKoRe13} proposed methods for testing the equality of (constant) mean functions of two strictly stationary functional time series. Once the null hypothesis, that $\{X_{t,T}\}$ is stationary, is rejected, their methods for testing the equality of mean functions cannot be applied to the observations. We extend their methods to nonstationary functional time series by using our main results.  

Assume that we have two samples of functional time series $\{X_{t,T_{1}}\}_{t=1}^{T_{1}}$ and $\{Y_{t,T_{2}}\}_{t=1}^{T_{2}}$ follow the location models
\begin{align*}
X_{t,T_{1}}(s) &= m_{1}\left({t \over T_{1}}, s\right) + v_{t,T_{1}}(s),\\ 
Y_{t,T_{2}}(s) &= m_{2}\left({t \over T_{2}}, s\right) + w_{t,T_{2}}(s),
\end{align*}
where $\{v_{t,T_{1}}\}$ and $\{w_{t,T_{2}}\}$ are zero-mean $H$-valued error processes that satisfy Assumption \ref{Ass-M2} and the associated stationary processes $\{v_{t}^{(u)}\}$ and $\{w_{t}^{(u)}\}$ have local long-run covariance kernel functions $c_{1,u}(s_{1},s_{2})$ and $c_{2,u}(s_{1},s_{2})$, which are defined analogously as $c^{(u)}$ in Theorem \ref{Thm2}, respectively. Assume that $m_{1}(u,\cdot)$ and $m_{2}(u, \cdot)$ are Lipschitz continuous with respect to the norm $\|\cdot\|$ as a function of $u$, i.e., there exists a positive constant $C_{m}$ such that $\|m_{j}(v,\cdot)-m_{j}(u,\cdot)\| \leq C_{m}|u-v|$ for $u,v \in [0,1]$, $j=1,2$. We also assume that $\{X_{t,T_{1}}\}_{t=1}^{T_{1}}$ and $\{Y_{t,T_{2}}\}_{t=1}^{T_{2}}$ are independent. Our interest is in the hypothesis of the form
\begin{align}\label{def: two-sample-test}
H_{0}:\ m_{1}(u,\cdot) = m_{2}(u,\cdot)\ \text{and}\ H_{1}:\ m_{1}(u,\cdot) \neq m_{2}(u,\cdot) 
\end{align}
in the space $H=L^{2}$ for fixed $u \in (0,1)$. The null hypothesis $H_{0}$ means that $\|m_{1}(u,\cdot) - m_{2}(u,\cdot)\| = 0$ and the alternative means that $\|m_{1}(u,\cdot) - m_{2}(u,\cdot)\| > 0$.

Define 
\begin{align*}
\bar{X}_{T_{1}}^{(u)} = {1 \over T_{1}h}\sum_{t=1}^{T_{1}}K_{1,h}\left(u - {t \over T_{1}}\right)X_{t,T_{1}},\ \bar{Y}_{T_{2}}^{(u)} = {1 \over T_{2}h}\sum_{t=1}^{T_{2}}K_{1,h}\left(u - {t \over T_{2}}\right)Y_{t,T_{2}}.
\end{align*}
An estimator of the $L^{2}$-norm between $m_{1}(u, \cdot)$ and $m_{2}(u,\cdot)$ is given by
\begin{align*}
\|\bar{X}_{T_{1}}^{(u)} - \bar{Y}_{T_{2}}^{(u)}\| &= \left(\int_{[0,1]^d} (\bar{X}_{T_{1}}^{(u)}(s) - \bar{Y}_{T_{2}}^{(u)}(s))^{2}ds\right)^{1/2}.
\end{align*}

The next result follows from Theorem \ref{Thm2} and the continuous mapping theorem. 
\begin{proposition}\label{Prop1}
Suppose $T_1h \wedge {1 \over Th^3} \to \infty$, ${T_{1} \over T_{1} + T_{2}} \to \theta \in (0,1)$ as $T_{1} \wedge T_{2} \to \infty$. Assume that $\{X_{t,T_{1}}\}$ and $\{Y_{t,T_{2}}\}$ satisfy the same assumption of Theorem \ref{Thm2}. Then under $H_{0}$, we have
\begin{align*}
U_{T_{1},T_{2}}^{(u)} &:= {T_{1}T_{2}h \over T_{1} + T_{2}}\|\bar{X}_{T_{1}}^{(u)} - \bar{Y}_{T_{2}}^{(u)}\|^{2} \stackrel{d}{\to} \|Z^{(u)}\|^{2} =: U_{\infty,\infty}^{(u)},
\end{align*}
where $Z^{(u)}$ is a Gaussian process with $E[Z^{(u)}]=0$ and $E[Z^{(u)}(s_{1})Z^{(u)}(s_{2})]= \phi^{(u)}(s_{1},s_{2})\int K^{2}_{1}(w)dw$ where $\phi^{(u)}(s_{1},s_{2}) = (1-\theta)c_{1,u}(s_{1},s_{2}) + \theta c_{2,u}(s_{1},s_{2})$. Additionally, under $H_{1}$, we have 
\begin{align*}
\|\bar{X}_{T_{1}}^{(u)} - \bar{Y}_{T_{2}}^{(u)}\|^{2} \stackrel{p}{\to} \|m_{1}(u,\cdot) - m_{2}(u,\cdot)\|^{2}. 
\end{align*}
Then the test statistics $U_{T_{1},T_{2}}^{(u)} \stackrel{p}{\to} \infty$ as $T_{1} \wedge T_{2} \to \infty$ under $H_{1}$.
\end{proposition}
Under $H_0$, the random variable $U_{\infty,\infty}^{(u)}$ has the representation 
\begin{align*}
U_{\infty,\infty}^{(u)} &\stackrel{d}{=} \sum_{j =1}^{\infty}\eta_{u,j}N_{j}^{2}, 
\end{align*}
where for random variables $A$ and $B$, $A\stackrel{d}{=}B$ denotes that $A$ and $B$ has the same distribution, $\{\eta_{u,j}\}$ is a sequence of the eigenvalues of the covariance operator $\Phi^{(u)}$ defined by $\phi^{(u)}$, and $\{N_{j}\}$ is a sequence of i.i.d. standard normal random variables.  
We can also estimate the covariance kernel function $\phi^{(u)}(s_{1},s_{2})$ by 
\begin{align*}
\hat{\phi}^{(u)}(s_{1},s_{2}) &= (1-\hat{\theta})\hat{c}_{1,u}(s_{1},s_{2}) + \hat{\theta}\hat{c}_{2,u}(s_{1},s_{2}),
\end{align*}
where $\hat{\theta} = {T_{1} \over T_{1}+T_{2}}$, and $\hat{c}_{j,u}(s_{1},s_{2})$, $j=1,2$ are defined analogously as $\hat{c}^{(u)}$ in (\ref{def: cov-kernel-est}). The next result establishes the consistency of $\hat{\phi}^{(u)}$. 
\begin{proposition}\label{Prop2}
Suppose $T_1h \wedge {1 \over Th^3} \to \infty$, ${T_{1} \over T_{1} + T_{2}} \to \theta \in (0,1)$ as $T_{1} \wedge T_{2} \to \infty$. Assume that $\{X_{t,T_{1}}\}$ and $\{Y_{t,T_{2}}\}$ satisfy the same assumption of Theorem \ref{Thm3}. Then we have
\begin{align*}
\int_{[0,1]^d} \int_{[0,1]^d} \left(\hat{\phi}^{(u)}(s_{1},s_{2}) - \phi^{(u)}(s_{1},s_{2})\right)^{2}ds_{1}ds_{2} \stackrel{p}{\to} 0.
\end{align*}
\end{proposition}

Since the distribution of $U_{\infty,\infty}^{(u)}$ depends on $\{\eta_{u,j}\}$, for the implementation of the two sample test (\ref{def: two-sample-test}), we approximate the distribution of $U_{\infty,\infty}^{(u)}$ using functional principal components. Consider the following truncated version of $U_{\infty,\infty}^{(u)}$:
\begin{align}
\bar{U}_{\infty,\infty}^{(u)} &= \sum_{j=1}^{q}\eta_{u,j}N_{j}^{2}. 
\end{align}
Let $\hat{\eta}_{u,j}$ be the $j$-th largest eigenvalue of $\hat{\Phi}^{(u)}$ defined by $\hat{\phi}^{(u)}$ and let $\hat{\nu}_{u,j}$ denote the eigenfunction associated with $\hat{\eta}_{u,j}$. Since $\hat{\Phi}^{(u)}$ and $\Phi^{(u)}$ are integral operators computed from covariance kernel functions $\hat{\phi}^{(u)}$ and $\phi^{(u)}$, Proposition \ref{Prop2} yields 
\begin{align}\label{two-sample-cov-op-conv}
\|\hat{\Phi}^{(u)} - \Phi^{(u)}\|_{\mathcal{S}} = \int_{[0,1]^d} \int_{[0,1]^d}\left(\hat{\phi}^{(u)}(s_{1},s_{2}) - \phi^{(u)}(s_{1},s_{2})\right)^{2}ds_{1}ds_{2} \stackrel{p}{\to} 0.
\end{align}
If $\eta_{u,1}>\cdots>\eta_{u,q}>\eta_{u,q+1}$, then (\ref{two-sample-cov-op-conv}) and the same argument in the proof of Corollary \ref{Cor1} yield
\begin{align}\label{two-sample-eigen-consistency}
\max_{1 \leq j \leq q}|\hat{\eta}_{u,j} - \eta_{u,j}| &\stackrel{p}{\to} 0,\ \max_{1 \leq j \leq q}\|\hat{\nu}'_{u,j} - \nu_{u,j}\| \stackrel{p}{\to} 0,
\end{align} 
where $\hat{\nu}'_{u,j}=\hat{\kappa}_{u,j}\hat{\nu}_{u,j}$ and $\hat{\kappa}_{u,j} = \text{sign}(\langle \hat{\nu}_{u,j}, \nu_{u,j} \rangle)$. 

%Consider a random vector $\hat{a}_{u} = (\hat{a}_{u,1},\dots,\hat{a}_{u,q})'$ where $\hat{a}_{u,j}= \langle \bar{X}_{T,1}^{(u)} - \bar{X}_{T,2}^{(u)}, \hat{\nu}_{u,j} \rangle$, $1 \leq j \leq q$. 
Define two kinds of truncated version of $U_{T_{1},T_{2}}^{(u)}$ based on functional principal components as follows:
\begin{align*}
\bar{U}_{T_{1},T_{2}}^{(u)} &= {T_{1}T_{2}h \over T_{1}+T_{2}}\sum_{j=1}^{q}\langle \bar{X}_{T_{1}}{(u)} - \bar{Y}_{T_{2}}^{(u)}, \hat{\nu}_{u,j} \rangle^{2},\ \tilde{U}_{T_{1},T_{2}}^{(u)} = {T_{1}T_{2}h \over T_{1}+T_{2}}\sum_{j=1}^{q}\hat{\eta}_{u,j}^{-1}\langle \bar{X}_{T_{1}}^{(u)} - \bar{Y}_{T_{2}}^{(u)}, \hat{\nu}_{u,j} \rangle^{2}.
\end{align*}
The next result characterizes asymptotic properties of $\bar{U}_{T_{1},T_{2}}^{(u)}$ and $\tilde{U}_{T_{1},T_{2}}^{(u)}$. It can be shown by using (\ref{two-sample-eigen-consistency}), Proposition \ref{Prop1}, and the continuous mapping theorem.
\begin{proposition}\label{Prop3}
Suppose $T_1h \wedge {1 \over Th^3} \to \infty$, ${T_{1} \over T_{1} + T_{2}} \to \theta \in (0,1)$ as $T_{1} \wedge T_{2} \to \infty$ and $\eta_{u,1}>\cdots>\eta_{u,q}>\eta_{u,q+1}$. Assume that $\{X_{t,T_{1}}\}$ and $\{Y_{t,T_{2}}\}$ satisfy the same assumption of Theorem \ref{Thm3}. Then under $H_{0}$, we have
\begin{align*}
\bar{U}_{T_{1},T_{2}}^{(u)} \stackrel{d}{\to} \sum_{j=1}^{q}\eta_{u,j}N_{j}^{2},\ \tilde{U}_{T_{1},T_{2}}^{(u)} \stackrel{d}{\to} \chi^{2}(q),
\end{align*}
where $\chi^{2}(q)$ is the $\chi^{2}$-distribution with $q$ degrees of freedom. Additionally, under $H_{1}$,   
\begin{align*}
{T_{1}+T_{2} \over T_{1}T_{2}h}\bar{U}_{T_{1},T_{2}}^{(u)} &\stackrel{p}{\to} \sum_{j=1}^{q}\langle m_{1}(u,\cdot)- m_{2}(u,\cdot), \nu_{u,j} \rangle^{2},\\ 
{T_{1}+T_{2} \over T_{1}T_{2}h}\tilde{U}_{T_{1},T_{2}}^{(u)} &\stackrel{p}{\to} \sum_{j=1}^{q}\eta_{u,j}^{-1}\langle m_{1}(u,\cdot)- m_{2}(u,\cdot), \nu_{u,j} \rangle^{2}.
\end{align*}
\end{proposition}
Proposition \ref{Prop3} implies that $\bar{U}_{T_{1},T_{2}}^{(u)}, \tilde{U}_{T_{1},T_{2}}^{(u)} \stackrel{p}{\to} \infty$ under $H_{1}$ if at least one of the projections $\langle m_{1}(u,\cdot)- m_{2}(u,\cdot), \nu_{u,j} \rangle$ is different from zero. 

When $\eta_{u,q_{0}}>\eta_{u,q_{0}+1}=0$ for some positive integer $q_{0}$, we can estimate $q_{0}$ based on a ratio method introduced in \cite{LaYa12}. The estimator $\hat{q}_{0}$ is given as follows: 
\begin{align*}
\hat{q}_{0} &= \argmin_{1 \leq j \leq \bar{q}}{\hat{\eta}_{u,j+1} \over \hat{\eta}_{u,j}},
\end{align*} 
where $\bar{q}$ is a prespecified positive number and set $0/0=1$ for convenience. Let $\varepsilon_{0}$ be a prespecified small positive number. We set $\left|\hat{\eta}_{u,j}/\hat{\eta}_{u,1}\right|$ as $0$ to reduce estimation error in the implementation, if its absolute value is smaller than $\varepsilon_{0}$. The following result establishes the consistency of $\hat{q}$. 

\begin{corollary}\label{Cor2}
Let $\bar{q} \geq q_{0}$. Under the same assumption of Proposition \ref{Prop2}, we have $\hat{q}_{0} \stackrel{p}{\to} q_{0}$ as $T \to \infty$. 
\end{corollary}
Applying the same rule as described above, it is straightforward to extend Corollary \ref{Cor2} to the case that $\eta_{u,q_{0}}$ and $\eta_{u,q_{0}+1}$ satisfy $\eta_{u,q_{0}}/\eta_{u,1} \geq \epsilon_{0}$ and $\eta_{u,q_{0}+1}/\eta_{u,1} < \epsilon_{0}$ for some prespecified small positive number $\epsilon_{0}$.

\section{Concluding remarks}

In this paper, we have studied a kernel method for estimating the time-varying covariance operator and the time-varying mean function of a locally stationary functional time series. We derived the convergence rate of the kernel estimator of the covariance operator and established a central limit theorem for the kernel-based locally weighted sample mean. As applications of these results, we extended methods for %prediction and 
testing the equality of time-varying mean functions, which were developed for stationary functional time series to our framework.  

To conclude, we shall mention two important future research topics. The first is to consider a data driven formula for selecting bandwidth parameters $h$ and $b$. This topic (at least on $b$) has not been addressed in previous works (e.g., \cite{HoKoRe13} and \cite{LiRoSh20}). Second, the two sample problem considered in this paper is limited to the hypothesis testing on the equality of the time-varying mean functions. We are hopeful that our theory can be extended to cover relevant hypotheses considered in the recent works by \cite{DeKoAu20} and \cite{DeKoVo20}.

\newpage

\appendix

\section{Proofs}\label{Appendix: proof}

\subsection{Proofs for Section \ref{Sec: Main}}

\begin{proof}[Proof of Proposition \ref{Prop: example}]
First, we check $X_{t}^{(u)}$ is a $L^4$-$m$-approximable process in $L_H^4$. Observe that
\begin{align*}
\|X_{t}^{(u)}-X_{m,t}^{(u)}\|^2 &= \sum_{k \geq 1}\left|H_k(u;\mathcal{F}_{t}^{(k)}) - H_k(u;\mathcal{F}_{t}^{(k,m)})\right|^2.
\end{align*} 
Then from Condition (E4), we obtain
\begin{align}
&\sum_{m \geq 1}v_4(X_{t}^{(u)}-X_{m,t}^{(u)})=\sum_{m \geq 1}E\left[\|X_{t}^{(u)}-X_{m,t}^{(u)}\|^4\right]^{1/4} \nonumber \\ 
&= \sum_{m \geq 1}E\left[\left(\sum_{k \geq 1}\left|H_k(u;\mathcal{F}_{t}^{(k)}) - H_k(u;\mathcal{F}_{t}^{(k,m)})\right|^2\right)^2\right]^{1/4} \nonumber \\
&\leq \sum_{m \geq 1}\left(\sum_{k_1 \geq 1}\sum_{k_2 \geq 1}E\left[\left|H_{k_1}(u;\mathcal{F}_{t}^{(k_1)}) - H_{k_1}(u;\mathcal{F}_{t}^{(k_1,m)})\right|^4\right]^{1/2} \right. \nonumber\\
&\left. \quad \quad \quad \times E\left[\left|H_{k_2}(u;\mathcal{F}_{t}^{(k_2)}) - H_{k_2}(u;\mathcal{F}_{t}^{(k_2,m)})\right|^4\right]^{1/2}\right)^{1/4} \nonumber \\
&=\sum_{m \geq 1}\left(\sum_{k \geq 1}E\left[\left(H_k(u;\mathcal{F}_{t,k}) - H_{k}(u;\mathcal{F}_{t,k}^{(m)})\right)^4\right]^{1/2}\right)^{1/2}<\infty. \label{example-L4-m-approx}
\end{align}
From Condition (E3), we also have
\begin{align}
E\left[\left\|X_{t}^{(u)}\right\|^4\right] &= E\left[\left(\sum_{k \geq 1}\left|H_k(u;\mathcal{F}_{t}^{(k)})\right|^2\right)^2\right] \nonumber \\
&= \sum_{k_1 \geq 1}\sum_{k_2 \geq 1}E\left[\left|H_{k_1}(u;\mathcal{F}_{t}^{(k_1)})\right|^2\left|H_{k_2}(u;\mathcal{F}_{t}^{(k_2)})\right|^2\right] \nonumber\\
&\leq \sum_{k_1 \geq 1}\sum_{k_2 \geq 1}E\left[\left|H_{k_1}(u;\mathcal{F}_{t}^{(k_1)})\right|^4\right]^{1/2}E\left[\left|H_{k_2}(u;\mathcal{F}_{t}^{(k_2)})\right|^4\right]^{1/2} \nonumber \\
&= \left(\sum_{k \geq 1}E\left[\left|H_{k}(u;\mathcal{F}_{t}^{(k)})\right|^4\right]^{1/2}\right)^{2}<\infty. \label{example-L4H}
\end{align}
Combining (\ref{example-L4-m-approx}) and (\ref{example-L4H}), $X_{t}^{(u)}$ is a $L^4$-$m$-approximable process in $L_H^4$. 

Next we check that $X_{t,T}$ is a locally stationary functional time series in $L^4_H$ that satisfies Definition \ref{def: LSfTS} with $\rho = 2$. This can be verified from Condition (E3) and 
\begin{align*}
\|X_{t,T}-X_{t}^{(u)}\|^2 &= \sum_{k \geq 1}\left|H_k(t/T;\mathcal{F}_{t}^{(k)}) - H_k(u;\mathcal{F}_{t}^{(k)})\right|^2 \leq \left|{t \over T} - u\right|^2 \sum_{k \geq 1}\left|\bar{H}_k(\mathcal{F}_{t}^{(k)})\right|^2.
\end{align*}
\end{proof}

\begin{proof}[Proof of Theorem \ref{Thm1}]
Define 
\begin{align*}
\tilde{C}_{u} &= {1 \over Th}\sum_{t=1}^{T}K_{1,h}\left(u - {t \over T}\right)X_{t}^{(u)} \otimes X_{t}^{(u)}.
\end{align*}

Decompose
\begin{align*}
\hat{C}_{u} - \tilde{C}_{u} &= {1 \over Th}\sum_{t=1}^{T}K_{1,h}\left(u - {t \over T}\right)\left( X_{t}^{(u)} \otimes (X_{t,T}-X_{t}^{(u)}) \right. \\
&\left. \quad + (X_{t,T}-X_{t}^{(u)}) \otimes X_{t}^{(u)} + (X_{t,T}-X_{t}^{(u)}) \otimes (X_{t,T}-X_{t}^{(u)}) \right)\\
&=: C_{u,1} + C_{u,2} + C_{u,3}. 
\end{align*}
For $C_{1,u}$, 
\begin{align*}
\|C_{1,u}(x)\| &\leq {1 \over Th}\sum_{t=1}^{T}K_{1,h}\left(u - {t \over T}\right)\|X_{t}^{(u)}\| \|x\| \|X_{t,T}-X_{t}^{(u)}\| \\
&\leq \left(h + {1 \over T}\right){\|x\| \over Th}\sum_{t=1}^{T}K_{1,h}\left(u - {t \over T}\right)\|X_{t}^{(u)}\| U_{t,T}^{(u)}. 
\end{align*}
Then
\begin{align*}
E[\|C_{1,u}\|_{\mathcal{L}}] &\leq \left(h + {1 \over T}\right){1 \over Th}\sum_{t=1}^{T}K_{1,h}\left(u - {t \over T}\right)E[\|X_{t}^{(u)}\|U_{t,T}^{(u)}]\\
&\leq \left(h + {1 \over T}\right){1 \over Th}\sum_{t=1}^{T}K_{1,h}\left(u - {t \over T}\right)(E[\|X_{t}^{(u)}\|^{2}])^{1/2}E[(U_{t,T}^{(u)})^{2}]^{1/2}\\ 
&= O\left(h + {1 \over T}\right). 
\end{align*}
This implies $\|C_{1,u}\|_{\mathcal{L}} = O_{P}(h + T^{-1})$. Likewise, we have 
\[
\|C_{2,u}\|_{\mathcal{L}} = O_{P}\left(h + {1 \over T}\right),\ \|C_{3,u}\|_{\mathcal{L}} = O_{P}\left(h^{2} + {1 \over T^{2}}\right).
\]
Therefore, 
\begin{align}\label{approx-CO1}
\|\hat{C}_{u} - \tilde{C}_{u}\|_{\mathcal{L}} = O_{P}\left(h + {1 \over T}\right).
\end{align}
Note that 
\begin{align*}
\tilde{C}_{u} - E[\tilde{C}_{u}] &= {1 \over Th}\sum_{t=1}^{T}K_{1,h}\left(u - {t \over T}\right)\left(X_{t}^{(u)} \otimes X_{t}^{(u)} - E[X_{t}^{(u)} \otimes X_{t}^{(u)}]\right).
\end{align*}
Define $B_{u,t}(y)=\langle X_{t}^{(u)}, y \rangle X_{t}^{(u)} - C_{u}(y)$ and 
$\tilde{B}_{u,t} = h^{-1}K_{1,h}(u-t/T)B_{u,t}$. Then 
\begin{align*}
Th\|\tilde{C}_{u} - E[\tilde{C}_{u}]\|_{\mathcal{S}}^{2} &= Th\left\|{1 \over T}\sum_{t=1}^{T}\tilde{B}_{u,k}\right\|_{\mathcal{S}}^{2}= {h \over T}\left(\sum_{t_{1}=1}^{T}\langle \tilde{B}_{u,t_{1}},\tilde{B}_{u,t_{2}} \rangle_{\mathcal{S}} + \sum_{t_{1} \neq t_{2}} \langle \tilde{B}_{u,t_{1}},\tilde{B}_{u,t_{2}} \rangle_{\mathcal{S}}\right).
\end{align*}
Note that $E[\langle B_{u,t_{1}},B_{u,t_{2}} \rangle_{\mathcal{S}}]=E[\langle B_{u,0},B_{u,t_{2}-t_{1}} \rangle_{\mathcal{S}}]$. Then we have
\begin{align*}
&{h \over T}E\left[\sum_{t_{1}=1}^{T}\langle \tilde{B}_{u,t_{1}},\tilde{B}_{u,t_{2}} \rangle_{\mathcal{S}} + \sum_{t_{1} \neq t_{2}} \langle \tilde{B}_{u,t_{1}},\tilde{B}_{u,t_{2}} \rangle_{\mathcal{S}}\right] \\
&\quad = {1 \over Th}\sum_{k=-(T-1)}^{T-1}\sum_{t_{1}=1+|k|}^{T}K_{1,h}\left(u-{t_{1} \over T}\right)K_{1,h}\left(u-{t_{1}-|k| \over T}\right)E\left[\langle B_{u,t_{1}}, B_{u,t_{1}-|k|} \rangle_{\mathcal{S}}\right]\\
&\quad = {1 \over Th}\sum_{k=-(T-1)}^{T-1}E\left[\langle B_{u,0}, B_{u,|k|} \rangle_{\mathcal{S}}\right]\sum_{t_{1}=1+|k|}^{T}K_{1,h}\left(u-{t_{1} \over T}\right)K_{1,h}\left(u-{t_{1}-|k| \over T}\right)\\
&\quad \leq {1 \over Th}\sum_{k=-(T-1)}^{T-1}\left|E\left[\langle B_{u,0}, B_{u,|k|} \rangle_{\mathcal{S}}\right]\right|\sum_{t_{1}=1+|k|}^{T}K_{1,h}\left(u-{t_{1} \over T}\right)K_{1,h}\left(u-{t_{1}-|k| \over T}\right)\\
&\quad \leq {2 \over Th}\sum_{k=0}^{T-1}\left|E\left[\langle B_{u,0}, B_{u,k} \rangle_{\mathcal{S}}\right]\right|\sum_{t_{1}=1+k}^{T}K_{1,h}\left(u-{t_{1} \over T}\right)K_{1,h}\left(u-{t_{1}-k \over T}\right).
\end{align*}
Hence we have
\begin{align*}
&ThE\left[\|\tilde{C}_{u} - E[\tilde{C}_{u}]\|_{\mathcal{S}}^{2}\right]\\
&\quad \leq {2 \over Th}\sum_{k=0}^{T-1}\left|E\left[\langle B_{u,0}, B_{u,k} \rangle_{\mathcal{S}}\right]\right|\sum_{t_{1}=1+k}^{T}K_{1,h}\left(u-{t_{1} \over T}\right)K_{1,h}\left(u-{t_{1}-k \over T}\right)\\
&\quad \lesssim {2 \over Th}\sum_{k=0}^{T-1}\left|E\left[\langle B_{u,0}, B_{u,k} \rangle_{\mathcal{S}}\right]\right|\sum_{t_{1}=1+k}^{T}K_{1,h}\left(u-{t_{1} \over T}\right)\\
&\quad \leq \left({2 \over Th}\sum_{t_{1}=1}^{T}K_{1,h}\left(u-{t_{1} \over T}\right)\right)\sum_{k=0}^{T-1}\left|E\left[\langle B_{u,0}, B_{u,k} \rangle_{\mathcal{S}}\right]\right|.
\end{align*}
Observe that
\begin{align*}
\langle B_{u,0}, B_{u,k} \rangle_{\mathcal{S}} &= \sum_{j =1}^{\infty}\langle B_{u,0}(v_{u,j}), B_{u,k}(v_{u,j}) \rangle \\
&= \langle X_{0}^{(u)}, X_{k}^{(u)} \rangle\sum_{j=1}^{\infty}\langle X_{0}^{(u)}, v_{u,j} \rangle\langle X_{k}^{(u)}, v_{u,j} \rangle - \sum_{j=1}^{\infty}\lambda_{u,j}\langle X_{0}^{(u)}, v_{u,j} \rangle^{2}\\ 
&\quad - \sum_{j=1}^{\infty}\lambda_{u,j}\langle X_{k}^{(u)}, v_{u,j} \rangle^{2} + \sum_{j = 1}^{\infty}\lambda_{u,j}^{2}\\
&= \langle X_{0}^{(u)}, X_{k}^{(u)} \rangle^{2} - \sum_{j=1}^{\infty}\lambda_{u,j}\langle X_{0}^{(u)}, v_{u,j} \rangle^{2} - \sum_{j=1}^{\infty}\lambda_{u,j}\langle X_{k}^{(u)}, v_{u,j} \rangle^{2} + \sum_{j = 1}^{\infty}\lambda_{u,j}^{2}.
\end{align*}
Then we have 
\begin{align}\label{B0-HS-norm}
E\left[\langle B_{u,0}, B_{u,k} \rangle_{\mathcal{S}}\right] &= E\left[ \langle X_{0}^{(u)}, X_{k}^{(u)} \rangle^{2}\right] - \sum_{j = 1}^{\infty}\lambda_{u,j}^{2},\ k \geq 1.
\end{align}
Recall that $X_{0}^{(u)}$ and $X_{k,k}^{(u)}$ are independent and identically distributed. Thus we have
\begin{align*}
E\left[\langle X_{0}^{(u)}, X_{k,k}^{(u)} \rangle^{2}\right] &= \sum_{j=1}^{\infty}\sum_{\ell=1}^{\infty}\left(E\left[\langle X_{0}^{(u)}, v_{u,j}\rangle\langle X_{0}^{(u)}, v_{u,\ell}\rangle\right]\right)^{2}\\
&=\sum_{j=1}^{\infty}\sum_{\ell=1}^{\infty}\left(\left\langle E\left[\langle X_{0}^{(u)}, v_{u,j}\rangle X_{0}^{(u)}\right], v_{u,\ell}\right\rangle\right)^{2}\\
&=\sum_{j=1}^{\infty}\sum_{\ell=1}^{\infty}\left(\langle C_{u}(v_{u,j}), v_{u,\ell}\rangle\right)^{2} = \sum_{j=1}^{\infty}\lambda_{u,j}^{2}.
\end{align*}
Moreover, since
\begin{align*}
\langle X_{0}^{(u)}, X_{k}^{(u)} - X_{k,k}^{(u)} \rangle^{2} & = \langle X_{0}^{(u)}, X_{k}^{(u)} \rangle^{2} - \langle X_{0}^{(u)}, X_{k,k}^{(u)} \rangle^{2} - 2\langle X_{0}^{(u)}, X_{k}^{(u)}-X_{k,k}^{(u)} \rangle\langle X_{0}^{(u)}, X_{k,k}^{(u)} \rangle,
\end{align*}
we have 
\begin{align}
E\left[ \langle X_{0}^{(u)}, X_{k}^{(u)} \rangle^{2}\right]- \sum_{j = 1}^{\infty}\lambda_{u,j}^{2} & = E\left[ \langle X_{0}^{(u)}, X_{k,k}^{(u)} \rangle^{2}\right]- \sum_{j = 1}^{\infty}\lambda_{u,j}^{2} + E\left[\langle X_{0}^{(u)}, X_{k}^{(u)} - X_{k,k}^{(u)} \rangle^{2}\right] \nonumber \\
&\quad + 2E\left[\langle X_{0}^{(u)}, X_{k}^{(u)}-X_{k,k}^{(u)} \rangle\langle X_{0}^{(u)}, X_{k,k}^{(u)} \rangle\right] \nonumber \\
& = E\left[\langle X_{0}^{(u)}, X_{k}^{(u)} - X_{k,k}^{(u)} \rangle^{2}\right] + 2E\left[\langle X_{0}^{(u)}, X_{k}^{(u)}-X_{k,k}^{(u)} \rangle\langle X_{0}^{(u)}, X_{k,k}^{(u)} \rangle\right]. \label{B0-HS-norm2}
\end{align}
Together with (\ref{B0-HS-norm}) and (\ref{B0-HS-norm2}), we have 
\begin{align*}
\left|E\left[\langle B_{u,0}, B_{u,k} \rangle_{\mathcal{S}}\right]\right| &\leq v_{4}^{2}(X_{0}^{(u)})v_{4}^{2}(X_{k}^{(u)}-X_{k,k}^{(u)}) + 2v_{4}^{3}(X_{0}^{(u)})v_{4}(X_{k}^{(u)}-X_{k,k}^{(u)})\\
&=v_{4}^{2}(X_{0}^{(u)})v_{4}^{2}(X_{1}^{(u)}-X_{k,1}^{(u)}) + 2v_{4}^{3}(X_{0}^{(u)})v_{4}(X_{1}^{(u)}-X_{k,1}^{(u)}).
\end{align*}
Since $X_{t}^{(u)}$ is $L^{4}$-$m$-approximable, for $u \in [C_{1}h, 1-C_{1}h]$, we have
\begin{align*}
&ThE\left[\|\tilde{C}_{u} - E[\tilde{C}_{u}]\|_{\mathcal{S}}^{2}\right]\\
&\quad \lesssim \left({2 \over Th}\sum_{t_{1}=1}^{T}K_{1,h}\left(u-{t_{1} \over T}\right)\right)\sum_{k=0}^{T-1}\left|E\left[\langle B_{u,0}, B_{u,k} \rangle_{\mathcal{S}}\right]\right|\\
&\quad \lesssim O(1) \times \left\{ v_{4}^{2}(X_{0}^{(u)})\sum_{k=1}^{\infty}v_{4}^{2}(X_{1}^{(u)}-X_{k,1}^{(u)}) + 2v_{4}^{3}(X_{0}^{(u)})\sum_{k=1}^{\infty}v_{4}(X_{1}^{(u)}-X_{k,1}^{(u)})\right\} <\infty.
\end{align*}
Therefore, we obtain
\begin{align}\label{approx-CO2}
\|\tilde{C}_{u} - E[\tilde{C}_{u}]\|_{\mathcal{S}} &= O_{P}\left(\sqrt{{1 \over Th}}\right).
\end{align}
Further, for each $u \in [C_{1}h, 1-C_{1}h]$, we have
\begin{align}\label{approx-CO3}
\|E[\tilde{C}_{u}] - C_{u}\|_{\mathcal{S}} &= \left|{1 \over Th}\sum_{t=1}^{T}K_{1,h}\left(u - {t \over T}\right)-1\right|\|C_{u}\|_{\mathcal{S}} \nonumber \\
&=\left|{1 \over Th}\sum_{t=1}^{T}K_{1,h}\left(u - {t \over T}\right)-1\right|\|C_{u}\|_{\mathcal{S}} \nonumber \\
&= O\left({1 \over Th^{2}}\right) + o(h).
\end{align}
For the third equality, we used Lemma \ref{lem: g}.

Combining (\ref{approx-CO1}), (\ref{approx-CO2}), (\ref{approx-CO3}) and $\|\cdot \|_{\mathcal{L}}\leq \|\cdot\|_{\mathcal{S}}$, we have
\begin{align*}
\|\hat{C}_{u} - C_{u}\|_{\mathcal{L}} &\leq \|\hat{C}_{u} - \tilde{C}_{u}\|_{\mathcal{L}} + \|\tilde{C}_{u} - E[\tilde{C}_{u}]\|_{\mathcal{S}} + \|E[\tilde{C}_{u}] - C_{u}\|_{\mathcal{S}}\\
&= O_{P}\left(h + {1 \over T}\right) + O_{P}\left(\sqrt{{1 \over Th}}\right) + O\left({1 \over Th^{2}}\right) + o(h)\\
&= O_{P}\left(h + \sqrt{{1 \over Th}} + {1 \over Th^{2}}\right)
\end{align*}
for each $u \in [C_{1}h, 1-C_{1}h]$. 
\end{proof}

\begin{proof}[Proof of Corollary \ref{Cor1}]

Combining Theorem \ref{Thm1} and Lemmas \ref{lem: HK-lem22} and \ref{lem: HK-lem23}, we have 
\begin{align*}
|\hat{\lambda}_{u,j} - \lambda_{u,j}| &\leq \|\hat{C}_{u} - C_{u}\|_{\mathcal{L}}= O_{P}\left(h + \sqrt{{1 \over Th}} + {1 \over Th^{2}}\right),\\
\|\hat{c}_{u,j}\hat{v}_{u,j} - v_{u,j}\| &\leq {2\sqrt{2} \over \alpha_{u,j}}\|\hat{C}_{u} - C_{u}\|_{\mathcal{L}}= O_{P}\left(h + \sqrt{{1 \over Th}} + {1 \over Th^{2}}\right),
\end{align*}
where $\alpha_{u,1} = \lambda_{u,1}-\lambda_{u,2}$, $\alpha_{u,j}=\min\{\lambda_{u,j-1}-\lambda_{u,j}, \lambda_{u,j} -\lambda_{u,j+1}\}$, $2 \leq j \leq q$. 
\end{proof}

\begin{proof}[Proof of Theorem \ref{Thm2}]

Define
\[
\tilde{X}_{T}^{(u)} = {1 \over Th}\sum_{t=1}^{T}K_{1,h}\left(u-{t \over T}\right)X_{t}^{(u)}.
\]
(Step1) In this step, we will show 
\begin{align}\label{AN}
ThE[\|\bar{X}_{T}^{(u)}-\tilde{X}_{T}^{(u)}\|^{2}] = o(1).
\end{align}
Observe that
\begin{align*}
\|\bar{X}_{T}^{(u)}-\tilde{X}_{T}^{(u)}\| &\leq {1 \over Th}\sum_{t=1}^{T}K_{1,h}\left(u-{t \over T}\right)\|X_{t,T}-X_{t}^{(u)}\| \\
&\leq  \left(h + {1 \over T}\right){1 \over Th}\sum_{t=1}^{T}K_{1,h}\left(u-{t \over T}\right)U_{t,T}^{(u)}.
\end{align*}
Then we have
\begin{align*}
ThE[\|\bar{X}_{T}^{(u)}-\tilde{X}_{T}^{(u)}\|^{2}] &\leq Th\left(h + {1 \over T}\right)^{2}{1 \over T^{2}h^{2}}E\left[\sum_{t_{1}=1}^{T}\sum_{t_{2}=1}^{T}K_{1,h}\left(u-{t_{1} \over T}\right)K_{1,h}\left(u-{t_{2} \over T}\right)U_{t_{1},T}^{(u)}U_{t_{2},T}^{(u)}\right]\\
&\lesssim Th^{3}{1 \over T^{2}h^{2}}\sum_{t_{1}=1}^{T}\sum_{t_{2}=1}^{T}K_{1,h}\left(u-{t_{1} \over T}\right)K_{1,h}\left(u-{t_{2} \over T}\right)E[(U_{t,T}^{(u)})^{2}]\\
&\lesssim Th^{3}\left({1 \over Th}\sum_{t=1}^{T}K_{1,h}\left(u-{t \over T}\right)\right)^{2} = o(1).
\end{align*}
This implies that
\begin{align*}
\sqrt{Th}\bar{X}_{T}^{(u)} &= \sqrt{Th}\tilde{X}_{T}^{(u)} + o_{P}(1).
\end{align*}
(Step2) In this step, we give a sketch of the lest of the proof. 

In Step3, we will show 
\begin{align}\label{conv1}
\limsup_{m \to \infty}\limsup_{T \to \infty}E\left[\int \left({1\over \sqrt{Th}}\sum_{t=1}^{T}K_{1,h}\left(u-{t \over T}\right)\left(X_{t}^{(u)}(s)-X_{m,t}^{(u)}(s)\right)\right)^{2}ds\right]=0
\end{align}
where the variables $X_{m,t}^{(u)}$ are defined in (\ref{m-approx-ver}).

In Step4, we will show that for any $m \geq 1$, 
\begin{align}\label{conv2}
{1 \over \sqrt{Th}}\sum_{t=1}^{T}K_{1,h}\left(u-{t \over T}\right)X_{m,t}^{(u)} \stackrel{d}{\to} G^{(u)}_{m}\ \text{in $L^{2}$} 
\end{align}
where $G^{(u)}_{m}$ is a Gaussian process with $E[G^{(u)}_{m}(s)]=0$ and with the covariance kernel function $E[G^{(u)}_{m}(s_{1})G^{(u)}_{m}(s_{2})]=c_{m}^{(u)}(s_{1},s_{2})\int K_{1}^{2}(z)dz$ where
\begin{align*}
c_{m}(s_{1},s_{2}) &= E[X_{0}^{(u)}(s_{1})X_{0}^{(u)}(s_{2})] + \sum_{t = 1}^{m}E[X_{0}^{(u)}(s_{1})X_{t}^{(u)}(s_{2})] + \sum_{t =1}^{m}E[X_{0}^{(u)}(s_{2})X_{t}^{(u)}(s_{1})]. 
\end{align*}

In Step5, we will show
\begin{align}\label{conv3}
G^{(u)}_{m} &\stackrel{d}{\to} G^{(u)}\ \text{in $L^{2}$}.
\end{align}

Combining Theorem 3.2 in \cite{Bi99} with (\ref{conv1})-(\ref{conv3}), we have
\begin{align}\label{asy-normal}
\sqrt{Th}\tilde{X}_{T}^{(u)} \stackrel{d}{\to} G^{(u)}\ \text{in $L^{2}$}. 
\end{align}

Therefore, (\ref{AN}) and (\ref{asy-normal}) yield $\sqrt{Th}\bar{X}_{T}^{(u)} \stackrel{d}{\to} G^{(u)}$ in $L^{2}$. 

(Step3) In this step, we will show (\ref{conv1}). By stationarity, 
\begin{align*}
&E\left[\left(\sum_{t=1}^{T}K_{1,h}\left(u-{t \over T}\right)\left(X_{t}^{(u)}(s)-X_{m,t}^{(u)}(s)\right)\right)^{2}\right]\\ 
&\quad =  \sum_{t=1}^{T}K_{1,h}^{2}\left(u-{t \over T}\right)E\left[(X_{1}^{(u)}(s)-X_{m,1}^{(u)}(s))^{2}\right]\\
&\quad \quad + 2\sum_{1 \leq t_{1}<t_{2}\leq T}K_{1,h}\left(u-{t_{1} \over T}\right)K_{1,h}\left(u-{t_{2} \over T}\right) E\left[(X_{t_{1}}^{(u)}(s)-X_{m,t_{1}}^{(u)}(s))(X_{t_{2}}^{(u)}(s)-X_{m,t_{2}}^{(u)}(s))\right].
\end{align*}
Note that
\begin{align}
&{1 \over Th}\int \sum_{t=1}^{T}K_{1,h}^{2}\left(u-{t \over T}\right)E\left[(X_{1}^{(u)}(s)-X_{m,1}^{(u)}(s))^{2}\right]ds \nonumber \\ 
&\quad = v_{2}^{2}(X_{1}^{(u)} - X_{m,1}^{(u)})\left({1 \over Th}\sum_{t=1}^{T}K_{1,h}^{2}\left(u-{t \over T}\right)\right) \sim v_{2}^{2}(X_{1}^{(u)} - X_{m,1}^{(u)}) \to 0\ \text{as $m \to \infty$}. \label{Step3-1}
\end{align}
Observe that if $t_{2}>t_{1}$, then $(X_{t_{1}}^{(u)}, X_{m,t_{1}}^{(u)})$ is independent of $X_{t_{2}-t_{1},t_{1}}^{(u)}$ because
\[
X_{t_{2}-t_{1},t_{2}}^{(u)} = f_{u}(\varepsilon_{t_{2}}, \varepsilon_{t_{2}-1},\dots, \varepsilon_{t_{1}+1}, \varepsilon_{t_{2},t_{1}}^{(t_{2}-t_{1})}, \varepsilon_{t_{2},t_{1}-1}^{(t_{2}-t_{1})},\dots).
\]
Then $E[(X_{t_{1}}^{(u)}(s)-X_{m,t_{1}}^{(u)}(s))X_{t_{2}-t_{1},t_{2}}^{(u)}]=0$ and so
\begin{align*}
&\sum_{1 \leq t_{1}<t_{2}\leq T}K_{1,h}\left(u-{t_{1} \over T}\right)K_{1,h}\left(u-{t_{2} \over T}\right)E\left[(X_{t_{1}}^{(u)}(s)-X_{m,t_{1}}^{(u)}(s))X_{t_{2}}^{(u)}(s)\right]\\
&\quad = \sum_{1 \leq t_{1}<t_{2}\leq T}K_{1,h}\left(u-{t_{1} \over T}\right)K_{1,h}\left(u-{t_{2} \over T}\right)E\left[(X_{t_{1}}^{(u)}(s)-X_{m,t_{1}}^{(u)}(s))(X_{t_{2}}^{(u)}(s)-X_{t_{2}-t_{1},t_{2}}^{(u)}(s))\right]. 
\end{align*}
Applying the Cauchy-Schwarz inequality, we have
\begin{align*}
&\left|\int \sum_{1 \leq t_{1}<t_{2}\leq T}K_{1,h}\left(u-{t_{1} \over T}\right)K_{1,h}\left(u-{t_{2} \over T}\right)E\left[(X_{t_{1}}^{(u)}(s)-X_{m,t_{1}}^{(u)}(s))(X_{t_{2}}^{(u)}(s)-X_{t_{2}-t_{1},t_{2}}^{(u)}(s))\right]ds\right|\\
&\quad \leq \sum_{1 \leq t_{1}<t_{2}\leq T}K_{1,h}\left(u-{t_{1} \over T}\right)K_{1,h}\left(u-{t_{2} \over T}\right) \\
&\quad \quad \times \int E\left[(X_{t_{1}}^{(u)}(s)-X_{m,t_{1}}^{(u)}(s))^{2}\right]^{1/2}E\left[(X_{t_{2}}^{(u)}(s)-X_{t_{2}-t_{1},t_{2}}^{(u)}(s))^{2}\right]^{1/2}ds\\
&\quad \leq \sum_{1 \leq t_{1}<t_{2}\leq T}K_{1,h}\left(u-{t_{1} \over T}\right)K_{1,h}\left(u-{t_{2} \over T}\right) \left(E[\|X_{t_{1}}^{(u)} - X_{m,t_{1}}^{(u)}\|^{2}]\right)^{1/2}\left(E[\|X_{t_{2}}^{(u)} - X_{t_{2}-t_{1},t_{2}}^{(u)}\|^{2}]\right)^{1/2} \\
&\quad = \sum_{1 \leq t_{1}<t_{2}\leq T}K_{1,h}\left(u-{t_{1} \over T}\right)K_{1,h}\left(u-{t_{2} \over T}\right) \left(E[\|X_{1}^{(u)} - X_{m,1}^{(u)}\|^{2}]\right)^{1/2}\left(E[\|X_{1}^{(u)} - X_{t_{2}-t_{1},1}^{(u)}\|^{2}]\right)^{1/2}\\
&\quad \lesssim  v_{2}(X_{1}^{(u)} - X_{m,1}^{(u)})\left(\sum_{t_{1}=1}^{T}K_{1,h}\left(u-{t_{1} \over T}\right)\right)\sum_{k \geq 1}v_{2}(X_{1}^{(u)} - X_{k,1}^{(u)}) \lesssim v_{2}(X_{1}^{(u)} - X_{m,1}^{(u)}) \times O\left(Th\right). 
\end{align*}
For the last inequality, we used the $L^{2}$-$m$-approximability of $X_{t}^{(u)}$. This yields
\begin{align*}
\limsup_{m \to \infty}\limsup_{T \to \infty}{1 \over Th}\left|\int \!\!\!\!\! \sum_{1 \leq t_{1}<t_{2}\leq T}\!\!\!\!\!K_{1,h}\left(u-{t_{1} \over T}\right)K_{1,h}\left(u-{t_{2} \over T}\right)E\left[(X_{t_{1}}^{(u)}(s)-X_{m,t_{1}}^{(u)}(s))X_{t_{2}}^{(u)}(s)\right]ds\right|=0.
\end{align*}
Likewise, we can show that
\begin{align*}
\limsup_{m \to \infty}\limsup_{T \to \infty}{1 \over Th}\left|\int \!\!\!\!\! \sum_{1 \leq t_{1}<t_{2}\leq T}\!\!\!\!\!K_{1,h}\left(u-{t_{1} \over T}\right)K_{1,h}\left(u-{t_{2} \over T}\right)E\left[(X_{t_{1}}^{(u)}(s)-X_{m,t_{1}}^{(u)}(s))X_{m,t_{2}}^{(u)}(s)\right]ds\right|=0.
\end{align*}
Combining these results and (\ref{Step3-1}) yield (\ref{conv1}). 

(Step4) In this step, we will show (\ref{conv2}). Recall that for every integer $m \geq 1$, $\{X_{m,t}^{(u)}\}$ is an $m$-dependent sequence of functions. Let $N>1$ be an integer and let $\{v_{m,j}\}$ and $\{\lambda_{m,j}\}$ are the orthonormal eigenfunctions and the corresponding eigenvalues of the integral operator with the kernel $c_{m}$. Then by the Karhunen-Lo\'eve expansion, we have
\[
X_{m,t}^{(u)}(s) = \sum_{j \geq 1}\langle X_{m,t}^{(u)}, v_{m,j} \rangle v_{m,j}(s).
\]
Define $X_{m,t}^{(u,N)}(s) =  \sum_{j = 1}^{N}\langle X_{m,t}^{(u)}, v_{m,j} \rangle v_{m,j}(s)$.
By the triangle inequality, we have
\begin{align*}
&\left\{E\left[\int\left(\sum_{t=1}^{T}K_{1,h}\left(u-{t \over T}\right)(X_{m,t}^{(u)}(s) - X_{m,t}^{(u,N)}(s))\right)^{2}ds\right]\right\}^{1/2}\\
&\quad \leq \left\{E\left[\int\left(\sum_{t \in I(0)}K_{1,h}\left(u-{t \over T}\right)(X_{m,t}^{(u)}(s) - X_{m,t}^{(u,N)}(s))\right)^{2}ds\right]\right\}^{1/2}  \\
&\quad \quad + \cdots + \left\{E\left[\int\left(\sum_{t \in I(m-1)}K_{1,h}\left(u-{t \over T}\right)(X_{m,t}^{(u)}(s) - X_{m,t}^{(u,N)}(s))\right)^{2}ds\right]\right\}^{1/2}
\end{align*}
where $I(k)=\{t: 1 \leq t \leq T, t=k\ (\text{mod}\ m)\}$, $0 \leq k \leq m-1$. Since $X_{m,t}^{(u)}$ is $m$-dependent, $\sum_{t \in I(\ell)}K_{1,h}\left(u-{t \over T}\right)(X_{m,t}^{(u)}(s) - X_{m,t}^{(u,N)}(s))$ is a sum of independent random variables. Thus we have
\begin{align*}
&{1 \over Th}E\left[\int\left(\sum_{t \in I(\ell)}K_{1,h}\left(u-{t \over T}\right)(X_{m,t}^{(u)}(s) - X_{m,t}^{(u,N)}(s))\right)^{2}ds\right]\\
&\quad =  {1 \over Th}\sum_{t \in I(\ell)}K_{1,h}^{2}\left(u-{t \over T}\right) E\left[\int (X_{m,t}^{(u)}(s) - X_{m,t}^{(u,N)}(s)))^{2}ds\right]\\
&\quad \leq {1 \over Th}\sum_{t=1}^{T}K_{1,h}^{2}\left(u-{t \over T}\right) \sum_{j >N}E[\langle X_{m,1}^{(u)}, v_{m,j}\rangle^{2}] \lesssim  \sum_{j >N}E[\langle X_{m,1}^{(u)}, v_{m,j}\rangle^{2}].
\end{align*}
Since $\lim_{N \to \infty}\sum_{j >N}E[\langle X_{m,1}^{(u)}, v_{m,j}\rangle^{2}] \to 0$, we conclude that for any $r>0$, 
\begin{align}\label{Step4-conv1}
\limsup_{N \to \infty}\limsup_{T \to \infty}P\left(\int \left({1 \over \sqrt{Th}}\sum_{t=1}^{T}K_{1,h}\left(u-{t \over T}\right)(X_{m,t}^{(u)}(s) - X_{m,t}^{(u,N)}(s))\right)^{2}ds>r\right) = 0. 
\end{align}
Observe that
\begin{align*}
{1 \over \sqrt{Th}}\sum_{t=1}^{T}K_{1,h}\left(u-{t \over T}\right)X_{m,t}^{(u,N)} &= \sum_{j=1}^{N}v_{m,j}{1 \over \sqrt{Th}}\sum_{t=1}^{T}K_{1,h}\left(u-{t \over T}\right)\langle X_{m,t}^{(u)}, v_{m,j} \rangle. 
\end{align*} 
Utilizing the orthogonality of $\{v_{m,j}\}$ and applying the central limit theorem for stationary $m$-dependent sequences (Theorem 6.4.2 in \cite{BrDa91}) and the Cram\'er-Wald theorem, we have
\begin{align*}
\left({1 \over \sqrt{Th}}\sum_{t=1}^{T}K_{1,h}\left(u-{t \over T}\right)\langle X_{m,t}^{(u)}, v_{m,j} \rangle, 1\leq j \leq N\right)' \stackrel{d}{\to} Z_{N}(\mathbf{0}, \Lambda_{N})
\end{align*}
where $Z_{N}(\mathbf{0}, \Lambda_{N})$ is a $N$-dimensional Gaussian random variable with zero mean and covariance matrix $\Lambda_{N} = \text{diag}(\lambda_{m,1},\dots, \lambda_{m,N})\int K_{1}^{2}(z)dz$. Then we have
\begin{align}\label{Spte4-conv2}
\sum_{j=1}^{N}v_{m,j}{1 \over \sqrt{Th}}\sum_{t=1}^{T}K_{1,h}\left(u-{t \over T}\right)\langle X_{m,t}^{(u)}, v_{m,j} \rangle \stackrel{d}{\to} \left(\int {K}_{1}^{2}(z)dz\right)^{1/2}\sum_{j=1}^{N}\lambda_{m,j}^{1/2}Z_{j}v_{m,j}\ \text{in $L^{2}$},
\end{align}
where $Z_{j}$ are independent standard Gaussian random variables. It is easy to see that
\begin{align}\label{Step4-conv3}
\int \left(\sum_{j > N}\lambda_{m,j}^{1/2}Z_{j}v_{m,j}\right)^{2}ds = \sum_{j > N}\lambda_{m,j}Z_{j}^{2} \stackrel{p}{\to} 0\ \text{as $N \to \infty$}. 
\end{align}

Combining Theorem 3.2 in \cite{Bi99} with (\ref{Step4-conv1})-(\ref{Step4-conv3}), we have
\[
{1 \over \sqrt{Th}}\sum_{t=1}^{T}K_{1,h}\left(u-{t \over T}\right)X_{m,t}^{(u)} \stackrel{d}{\to} \sum_{j =1}^{\infty}\lambda_{m,j}^{1/2}Z_{j}v_{m,j}\ \text{in $L^{2}$} 
\]
for any $m \geq 1$. Since $\sum_{j =1}^{\infty}\lambda_{m,j}^{1/2}Z_{j}v_{m,j}$ has the same distribution as $G^{(u)}_{m}$, we obtain (\ref{conv2}).

(Step5) In this step, we will show (\ref{conv3}). Since $G^{(u)}_{m}$ is a zero-mean Gaussian process, it is sufficient to show 
\begin{align}\label{c-cm-conv}
\int \int (c_{m}^{(u)}(s_{1},s_{2}) - c^{(u)}(s_{1},s_{2}))^{2}ds_{1}ds_{2} \to \infty\ \text{as $m \to \infty$}.
\end{align} 
First, we show that the kernel $c^{(u)}$ converge in $L^{2}([0,1]^d \times [0,1]^d)$. Observe that 
\begin{align*}
(c_{m}^{(u)}(s_{1},s_{2}))^{2} &\leq 4\left(\left(E[X_{0}^{(u)}(s_{1})X_{0}^{(u)}(s_{2})]\right)^{2} + \left(\sum_{t=1}^{m}E[X_{0}^{(u)}(s_{1})X_{t}^{(u)}(s_{2})]\right)^{2} \right. \\
&\left. \quad \quad + \left(\sum_{t=1}^{m}E[X_{0}^{(u)}(s_{2})X_{t}^{(u)}(s_{1})]\right)^{2}\right)\\
&=: 4\left(C_{m,1}(s_{1},s_{2}) + C_{m,2}(s_{1},s_{2}) + C_{m,3}(s_{1},s_{2})\right) =: \bar{c}_{m}(s_{1},s_{2}). 
\end{align*}
For $C_{m,1}(s_{1},s_{2})$, 
\begin{align}\label{cm-bound1}
\int \int C_{m,1}(s_{1},s_{2}) ds_{1}ds_{2} &\leq \int \int E[(X_{0}^{(u)}(s_{1}))^{2}]E[(X_{0}^{(u)}(s_{2}))^{2}]ds_{1}ds_{2}= v_{2}^{4}(X_{0}^{(u)}). 
\end{align}
For $C_{m,2}(s_{1},s_{2})$, applying the Cauchy-Schwarz inequality yields
\begin{align}
&\int \int C_{m,2}(s_{1},s_{2}) ds_{1}ds_{2} \nonumber \\ 
&\quad \leq \sum_{t_{1}=1}^{m}\sum_{t_{2}=1}^{m}\int \int \left|E[X_{0}^{(u)}(s_{1})X_{t_{1}}^{(u)}(s_{2})]\right|\left|E[X_{0}^{(u)}(s_{1})X_{t_{2}}^{(u)}(s_{2})]\right| ds_{1}ds_{2} \nonumber \\
&\quad = \sum_{t_{1}=1}^{m}\sum_{t_{2}=1}^{m}\int \int \left|E[X_{0}^{(u)}(s_{1})(X_{t_{1}}^{(u)}(s_{2}) - X_{t_{1},t_{1}}^{(u)}(s_{2}))]\right|\left|E[X_{0}^{(u)}(s_{1})(X_{t_{2}}^{(u)}(s_{2}) - X_{t_{2},t_{2}}^{(u)}(s_{2}))]\right| ds_{1}ds_{2}\nonumber \\
&\quad \leq  \sum_{t_{1}=1}^{m}\sum_{t_{2}=1}^{m}\int \int E[(X_{0}^{(u)}(s_{1}))^{2}]E[(X_{t_{1}}^{(u)}(s_{2}) - X_{t_{1},t_{1}}^{(u)}(s_{2}))^{2}]^{1/2}E[(X_{t_{2}}^{(u)}(s_{2}) - X_{t_{2},t_{2}}^{(u)}(s_{2}))^{2}]^{1/2} ds_{1}ds_{2} \nonumber \\
&\quad \leq  v_{2}^{2}(X_{0}^{(u)})\sum_{t_{1}=1}^{m}\sum_{t_{2}=1}^{m}v_{2}(X_{0}^{(u)} - X_{t_{1},0}^{(u)})v_{2}(X_{0}^{(u)} - X_{t_{2},0}^{(u)}) \nonumber \\
&\quad \leq v_{2}^{2}(X_{0}^{(u)})\left(\sum_{t=1}^{\infty}v_{2}(X_{0}^{(u)} - X_{t,0}^{(u)})\right)^{2}<\infty. \label{cm-bound2}
\end{align}
Likewise, 
\begin{align}\label{cm-bound3}
\int \int C_{m,3}(s_{1},s_{2}) ds_{1}ds_{2} \leq v_{2}^{2}(X_{0}^{(u)})\left(\sum_{t=1}^{\infty}v_{2}(X_{0}^{(u)} - X_{t,0}^{(u)})\right)^{2}.
\end{align}
Then we have $\lim_{m \to \infty}\int \int \bar{c}_{m}^{(u)}(s_{1},s_{2})ds_{1}ds_{2}<\infty$.

Combining (\ref{cm-bound1})-(\ref{cm-bound3}) and applying the monotone convergence theorem, we have
\begin{align}
\int \int (c^{(u)}(s_{1},s_{2}))^{2}ds_{1}ds_{2} &= \int \int \lim_{m \to \infty}(c_{m}^{(u)}(s_{1},s_{2}))^{2}ds_{1}ds_{2} \nonumber \\
&\leq \int \int \lim_{m \to \infty}\bar{c}_{m}^{(u)}(s_{1},s_{2})ds_{1}ds_{2} \nonumber \\
&=\lim_{m \to \infty}\int \int \bar{c}_{m}^{(u)}(s_{1},s_{2})ds_{1}ds_{2}<\infty. \label{c-bound}
\end{align}
Applying almost the same argument to show (\ref{c-bound}), we can show (\ref{c-cm-conv}). Therefore, we complete the proof. 
\end{proof}

\begin{proof}[Proof of Theorem \ref{Thm3}]

(Step1) In this step, we give a sketch of the proof.  

In Step2, we will show 
\begin{align}\label{Step1-01}
\int \int \left(\hat{\gamma}_{0}^{(u)}(s_{1},s_{2}) - E[X_{0}^{(u)}(s_{1})X_{0}^{(u)}(s_{2})]\right)^{2}ds_{1}ds_{2} &= o_{P}(1).
\end{align}

Define 
\begin{align*}
\gamma_{t,11}(s_{1},s_{2}) &= {1 \over Th}\sum_{j=t+1}^{T}K_{1,h}\left(u-{j \over T}\right)X_{j}^{(u)}(s_{1})X_{j-t}^{(u)}(s_{2}).
\end{align*}

In Step3, we will show 
\begin{align}\label{Step1-1}
\int \int \left(\sum_{t=1}^{T-1}K_{2}(t/b)\hat{\gamma}_{t}^{(u)}(s_{1},s_{2}) - \sum_{t=1}^{T-1}K_{2}(t/b)\gamma_{t,11}(s_{1},s_{2})\right)^{2}ds_{1}ds_{2} &= o_{P}(1).
\end{align}

In Step4, we will show 
\begin{align}\label{Step1-2}
\int \int \left(\sum_{t=1}^{T-1}K_{2}(t/b)\gamma_{t,11}(s_{1},s_{2}) - c_{1}^{(u)}(s_{1},s_{2})\right)^{2}ds_{1}ds_{2} = o_{P}(1),
\end{align}
where $c_{1}^{(u)}(s_{1},s_{2}) = \sum_{t=1}^{\infty}E[X_{0}^{(u)}(s_{2})X_{t}^{(u)}(s_{1})]$. Hence, (\ref{Step1-1}) and (\ref{Step1-2}) yields
\begin{align}\label{Step1-02}
\int \int \left(\sum_{t=1}^{T-1}K_{2}(t/b)\hat{\gamma}_{t}^{(u)}(s_{1},s_{2}) - c_{1}^{(u)}(s_{1},s_{2})\right)^{2}ds_{1}ds_{2} = o_{P}(1). 
\end{align}

Likewise, we can show 
\begin{align}\label{Step1-03}
\int \int \left(\sum_{t=1}^{T-1}K_{2}(t/b)\hat{\gamma}_{t}^{(u)}(s_{2},s_{1}) - c_{1}^{(u)}(s_{2},s_{1})\right)^{2}ds_{1}ds_{2} = o_{P}(1). 
\end{align}
Combining (\ref{Step1-01}), (\ref{Step1-02}) and (\ref{Step1-03}), we complete the proof. 

(Step2) In this step, we will show 
\begin{align}\label{cov-kernel-conv1}
\int \int \left(\hat{\gamma}_{0}^{(u)}(s_{1},s_{2}) - E[X_{0}^{(u)}(s_{1})X_{0}^{(u)}(s_{2})]\right)^{2}ds_{1}ds_{2} &= o_{P}(1).
\end{align}

Decompose
\begin{align*}
\hat{\gamma}_{0}^{(u)}(s_{1},s_{2}) &= {1 \over Th}\sum_{t=1}^{T}K_{1,h}\left(u-{t \over T}\right)\left(X_{t,T}(s_{1}) - \bar{X}_{T}^{(u)}(s_{1})\right)\left(X_{t,T}(s_{2}) - \bar{X}_{T}^{(u)}(s_{2})\right) \\
&={1 \over Th}\sum_{t=1}^{T}K_{1,h}\left(u-{t \over T}\right)X_{t,T}(s_{1})X_{t,T}(s_{2}) - \bar{X}_{T}^{(u)}(s_{1})\bar{X}_{T}^{(u)}(s_{2})\\
&= {1 \over Th}\sum_{t=1}^{T}K_{1,h}\left(u-{t \over T}\right)X_{t}^{(u)}(s_{1})X_{t}^{(u)}(s_{2})\\
&\quad + {1 \over Th}\sum_{t=1}^{T}K_{1,h}\left(u-{t \over T}\right)X_{t}^{(u)}(s_{1})(X_{t,T}(s_{2})-X_{t}^{(u)}(s_{2}))\\
&\quad + {1 \over Th}\sum_{t=1}^{T}K_{1,h}\left(u-{t \over T}\right)(X_{t,T}(s_{1})-X_{t}^{(u)}(s_{1}))X_{t}^{(u)}(s_{2})\\
&\quad + {1 \over Th}\sum_{t=1}^{T}K_{1,h}\left(u-{t \over T}\right)(X_{t,T}(s_{1})-X_{t}^{(u)}(s_{1}))(X_{t,T}(s_{2})-X_{t}^{(u)}(s_{2}))\\
&\quad - \bar{X}_{T}^{(u)}(s_{1})\bar{X}_{T}^{(u)}(s_{2})\\
&=: \gamma_{0,1}(s_{1},s_{2}) + \gamma_{0,2}(s_{1},s_{2}) + \gamma_{0,3}(s_{1},s_{2})  + \gamma_{0,4}(s_{1},s_{2}) -\bar{X}_{T}^{(u)}(s_{1})\bar{X}_{T}^{(u)}(s_{2}). 
\end{align*}

Define 
\begin{align*}
w_{T}^{(u)} = {1 \over Th}\sum_{t=1}^{T}K_{1,h}\left(u - {t \over T}\right),\ w_{t,T}^{(u)} = {K_{1,h}(u-t/T) \over \sum_{t=1}^{T}K_{1,h}(u-t/T)}. 
\end{align*}
Note that $\sum_{t=1}^{T}w_{t,T}^{(u)} = 1$. Observe that
\begin{align*}
&\int \int \left\{\hat{\gamma}_{0}^{(u)}(s_{1},s_{2}) - E[X_{0}^{(u)}(s_{1})X_{0}^{(u)}(s_{2})]\right\}^{2}ds_{1}ds_{2} \\
&\quad \leq 32\int \int \left\{\gamma_{0,1}^{(u)}(s_{1},s_{2}) - w_{T}^{(u)}E[X_{0}^{(u)}(s_{1})X_{0}^{(u)}(s_{2})]\right\}^{2}ds_{1}ds_{2} + 32\int \int \left(\gamma_{0,2}^{(u)}(s_{1},s_{2})\right)^{2}ds_{1}ds_{2}\\
&\quad \quad + 32\int \int \left(\gamma_{0,3}^{(u)}(s_{1},s_{2})\right)^{2}ds_{1}ds_{2} + 32\int \int \left(\gamma_{0,4}^{(u)}(s_{1},s_{2})\right)^{2}ds_{1}ds_{2}\\
&\quad \quad + 32\int \int \left((1-w_{T}^{(u)})E[X_{0}^{(u)}(s_{1})X_{0}^{(u)}(s_{2})]\right)^{2}ds_{1}ds_{2} + 32\int \int \left(\bar{X}_{T}^{(u)}(s_{1})\bar{X}_{T}^{(u)}(s_{2})\right)^{2}ds_{1}ds_{2}\\
&\quad =: \sum_{j=1}^{6}\Gamma_{j,T}^{(u)}.
\end{align*}
For $\Gamma_{6,T}^{(u)}$, applying Theorem \ref{Thm2}, we have
\begin{align}\label{Gamma-bound1}
\Gamma_{6,T}^{(u)} &= 32\left(\int \left(\bar{X}_{T}^{(u)}(s)\right)^{2}ds \right)^{2} = \|\bar{X}_{T}^{(u)}\|^{4} = O_{P}\left({1 \over T^{2}h^{2}}\right). 
\end{align} 
For $\Gamma_{5,T}^{(u)}$, applying Lemma \ref{lem: g}, we have
\begin{align}
\Gamma_{5,T}^{(u)} &= 32(1-w_{T}^{(u)})^{2}\int \int \left(E[X_{0}^{(u)}(s_{1})X_{0}^{(u)}(s_{2})]\right)^{2}ds_{1}ds_{2} \nonumber \\
&\leq 32(1-w_{T}^{(u)})^{2}\left(E\left[\int (X_{0}^{(u)}(s))^{2}ds\right]\right)^{2} \lesssim o\left(h^{2}\right). \label{Gamma-bound2}
\end{align}
For $\Gamma_{2,T}^{(u)}$, observe that
\begin{align*}
&E\left[\Gamma_{2,T}^{(u)}\right]\\ 
&\quad \lesssim {1 \over (Th)^{2}}\sum_{t=1}^{T}K_{1,h}^{2}\left(u- {t \over T}\right)E\left[\left(\int (X_{t}^{(u)}(s_{1}))^{2}ds_{1}\right)\left(\int (X_{t,T}(s_{2})-X_{t}^{(u)}(s_{2}))^{2}ds_{2} \right)\right]\\
&\quad \quad+ {1 \over (Th)^{2}}\sum_{t_{1}=1}^{T}\sum_{t_{2}=1}^{T}K_{1,h}\left(u- {t_{1} \over T}\right)K_{1,h}\left(u- {t_{2} \over T}\right)\\
&\quad \quad \quad \times \left|E\left[\left(\int X_{t_{1}}^{(u)}(s_{1})X_{t_{2}}^{(u)}(s_{1})ds_{1} \right)\left(\int (X_{t_{1},T}(s_{2})-X_{t_{1}}^{(u)}(s_{2}))(X_{t_{2},T}(s_{1})-X_{t_{2}}^{(u)}(s_{1}))ds_{2} \right)\right]\right|.
\end{align*}
Note that
\begin{align*}
E\left[\left(\int (X_{t}^{(u)}(s_{1}))^{2}ds_{1}\right)\left(\int (X_{t,T}(s_{2})-X_{t}^{(u)}(s_{2}))^{2}ds_{2} \right)\right] &= E\left[\|X_{t}^{(u)}\|^{2}\|X_{t,T}-X_{t}^{(u)}\|^{2}\right]\\
&\leq \left(h + {1 \over T}\right)^{2}E\left[\|X_{t}^{(u)}\|^{2}(U_{t,T}^{(u)})^{2}\right]\\
&\lesssim h^{2}v_{4}^{2}(X_{0}^{(u)})
\end{align*}
and 
\begin{align*}
&\left|E\left[\left(\int X_{t_{1}}^{(u)}(s_{1})X_{t_{2}}^{(u)}(s_{1})ds_{1} \right)\left(\int (X_{t_{1},T}(s_{2})-X_{t_{1}}^{(u)}(s_{2}))(X_{t_{2},T}(s_{1})-X_{t_{2}}^{(u)}(s_{1}))ds_{2} \right)\right]\right|\\
&\quad \leq E\left[\|X_{t_{1}}^{(u)}\| \|X_{t_{2}}^{(u)}\| \|X_{t_{1},T}-X_{t_{1}}^{(u)}\| \|X_{t_{2},T}-X_{t_{2}}^{(u)}\| \right]\\
&\quad \leq \left(h + {1 \over T}\right)^{2}E\left[\|X_{t_{1}}^{(u)}\| \|X_{t_{2}}^{(u)}\| U_{t_{1},T}^{(u)}U_{t_{2},T}^{(u)}\right] \lesssim h^{2}v_{4}^{2}(X_{0}^{(u)}).
\end{align*}
Then we have
\begin{align}\label{Gamma-bound3}
E\left[\Gamma_{2,T}^{(u)}\right] &\lesssim {h^{2} \over Th}\left({1 \over Th}\sum_{t=1}^{T}K_{1,h}^{2}\left(u- {t \over T}\right)\right) + h^{2}\left({1 \over Th}\sum_{t=1}^{T}K_{1,h}\left(u- {t \over T}\right)\right)^{2} \lesssim h^{2}. 
\end{align}
Likewise, we have 
\begin{align}\label{Gamma-bound4}
E[\Gamma_{3,T}^{(u)}] \lesssim h^{2},\ E[\Gamma_{4,T}^{(u)}] \lesssim h^{4}.
\end{align}

For $\Gamma_{1,T}^{(u)}$, applying the ergodic theorem for weighted random variables in a Hilbert space (see \cite{HaPl69} for example), 
\begin{align}\label{Gamma-bound5}
\Gamma_{1,T}^{(u)} &= 32\left(w_{T}^{(u)}\right)^{2}\int \int \left\{(w_{T}^{(u)})^{-1}\gamma_{0,1}^{(u)}(s_{1},s_{2}) - E[X_{0}^{(u)}(s_{1})X_{0}^{(u)}(s_{2})]\right\}^{2}ds_{1}ds_{2} = o_{P}(1). 
\end{align}
Combing (\ref{Gamma-bound1})-(\ref{Gamma-bound5}), we obtain (\ref{cov-kernel-conv1}).

(Step3)
Define 
\begin{align*}
\gamma_{t,11}(s_{1},s_{2}) &= {1 \over Th}\sum_{j=t+1}^{T}K_{1,h}\left(u-{j \over T}\right)X_{j}^{(u)}(s_{1})X_{j-t}^{(u)}(s_{2}).
\end{align*}
In this step, we will show 
\begin{align}\label{cov-kernel-conv2}
\int \int \left(\sum_{t=1}^{T-1}K_{2}(t/b)\hat{\gamma}_{t}^{(u)}(s_{1},s_{2}) - \sum_{t=1}^{T-1}K_{2}(t/b)\gamma_{t,11}(s_{1},s_{2})\right)^{2}ds_{1}ds_{2} &= o_{P}(1).
\end{align}
Observe that 
\begin{align*}
\hat{\gamma}_{t}^{(u)}(s_{1},s_{2}) &= {1 \over Th}\sum_{j=t+1}^{T}K_{1,h}\left(u-{j \over T}\right)X_{j,T}(s_{1})X_{j-t,T}(s_{2})\\
&\quad - \left\{{1 \over Th}\sum_{j=t+1}^{T}K_{1,h}\left(u-{j \over T}\right)X_{j,T}(s_{1})\right\}\bar{X}_{T}^{(u)}(s_{2})\\
&\quad - \left\{{1 \over Th}\sum_{j=t+1}^{T}K_{1,h}\left(u-{j \over T}\right)X_{j-t,T}(s_{2})\right\}\bar{X}_{T}^{(u)}(s_{1})\\
&\quad + \left\{{1 \over Th}\sum_{j=t+1}^{T}K_{1,h}\left(u-{j \over T}\right)\right\}\bar{X}_{T}^{(u)}(s_{1})\bar{X}_{T}^{(u)}(s_{2})\\
&=: \gamma_{t,1}(s_{1},s_{2}) -  \gamma_{t,2}(s_{1},s_{2}) -  \gamma_{t,3}(s_{1},s_{2}) +  \gamma_{t,4}(s_{1},s_{2}).
\end{align*}

(Step3-1) For $\gamma_{t,1}(s_{1},s_{2})$, decompose
\begin{align*}
\gamma_{t,1}(s_{1},s_{2}) &= {1 \over Th}\sum_{j=t+1}^{T}K_{1,h}\left(u-{j \over T}\right)X_{j}^{(u)}(s_{1})X_{j-t}^{(u)}(s_{2})\\
&\quad + {1 \over Th}\sum_{j=t+1}^{T}K_{1,h}\left(u-{j \over T}\right)(X_{j,T}(s_{1})-X_{j}^{(u)}(s_{1}))X_{j-t}^{(u)}(s_{2})\\
&\quad + {1 \over Th}\sum_{j=t+1}^{T}K_{1,h}\left(u-{j \over T}\right)(X_{j-t,T}(s_{2})-X_{j-t}^{(u)}(s_{2}))X_{j}^{(u)}(s_{1})\\
&\quad + {1 \over Th}\sum_{j=t+1}^{T}K_{1,h}\left(u-{j \over T}\right)(X_{j,T}(s_{1})-X_{j}^{(u)}(s_{1}))(X_{j-t,T}(s_{2})-X_{j-t}^{(u)}(s_{2}))\\
&=: \gamma_{t,11}(s_{1},s_{2}) + \gamma_{t,12}(s_{1},s_{2}) + \gamma_{t,13}(s_{1},s_{2}) + \gamma_{t,14}(s_{1},s_{2}).
\end{align*}
For $\gamma_{t,11}(s_{1},s_{2})$, using the triangle inequality, we have
\begin{align*}
&E\left[\left(\int \int \left(\sum_{t=1}^{T-1}K_{2}(t/b)\gamma_{t,12}(s_{1},s_{2})\right)^{2}ds_{1}ds_{2}\right)^{1/2}\right]\\ 
&\quad \leq \sum_{t=1}^{T-1}|K_{2}(t/b)|E\left[\left(\int \int \left(\gamma_{t,12}(s_{1},s_{2})\right)^{2}ds_{1}ds_{2}\right)^{1/2}\right].
\end{align*}
Since we have
\begin{align*}
&E\left[\int \int \left(\gamma_{t,12}(s_{1},s_{2})\right)^{2}ds_{1}ds_{2}\right]\\ 
&\quad \leq {1 \over (Th)^{2}}\sum_{j=t+1}^{T}K_{1,h}^{2}\left(u-{j \over T}\right)E\left[\int\int (X_{j,T}(s_{1})-X_{j}^{(u)}(s_{1}))^{2}(X_{j-t}^{(u)}(s_{2}))^{2}ds_{1}ds_{2} \right]\\
&\quad \quad  + {1 \over (Th)^{2}}\sum_{j_{1}=t+1}^{T}\sum_{j_{2}=t+1}^{T}K_{1,h}\left(u-{j_{1} \over T}\right)K_{1,h}\left(u-{j_{2} \over T}\right)\\
&\quad \quad \quad \times \left|E\left[\int\int (X_{j_{1},T}(s_{1})-X_{j_{1}}^{(u)}(s_{1})) (X_{j_{2},T}(s_{1})-X_{j_{2}}^{(u)}(s_{1}))X_{j_{1}-t}^{(u)}(s_{2})X_{j_{2}-t}^{(u)}(s_{2})ds_{1}ds_{2} \right]\right| \\
&\quad \lesssim {h^{2} \over Th}\left({1 \over Th}\sum_{t=1}^{T}K_{1,h}^{2}\left(u- {t \over T}\right)\right) + h^{2}\left({1 \over Th}\sum_{t=1}^{T}K_{1,h}\left(u- {t \over T}\right)\right)^{2} \lesssim h^{2},
\end{align*}
we then have
\begin{align}\label{ga-t12}
E\left[\int \int \left(\sum_{t=1}^{T-1}K_{2}(t/b)\gamma_{t,12}(s_{1},s_{2})\right)^{2}ds_{1}ds_{2}\right] &\lesssim h^{2}\sum_{t=1}^{T-1}|K_{2}(t/b)| \lesssim h^{2}b \to 0.
\end{align}
Likewise, we can show 
\begin{align}\label{ga-t13}
E\left[\int \int \left(\sum_{t=1}^{T-1}K_{2}(t/b)\gamma_{t,13}(s_{1},s_{2})\right)^{2}ds_{1}ds_{2}\right] &\lesssim h^{2}\sum_{t=1}^{T-1}|K_{2}(t/b)| \lesssim h^{2}b \to 0
\end{align}
and 
\begin{align}\label{ga-t14}
E\left[\int \int \left(\sum_{t=1}^{T-1}K_{2}(t/b)\gamma_{t,14}(s_{1},s_{2})\right)^{2}ds_{1}ds_{2}\right] &\lesssim h^{2}\sum_{t=1}^{T-1}|K_{2}(t/b)| \lesssim h^{4}b \to 0.
\end{align}
%Combining (\ref{ga-t12})-(\ref{ga-t14}), we have 
%\begin{align*}
%\int \int \left(\sum_{t=1}^{T-1}K_{2}(t/b)\gamma_{t,1}(s_{1},s_{2})\right)^{2}ds_{1}ds_{2} &= \int \int \left(\sum_{t=1}^{T-1}K_{2}(t/b)\gamma_{t,11}(s_{1},s_{2})\right)^{2}ds_{1}ds_{2} + o_{P}(h^{2}b).
%\end{align*}

(Step3-2) For $\gamma_{t,2}(s_{1},s_{2})$, decompose
\begin{align*}
\gamma_{t,2}(s_{1},s_{2}) &=\left\{{1 \over Th}\sum_{j=t+1}^{T}K_{1,h}\left(u-{j \over T}\right)X_{j}^{(u)}(s_{1})\right\}\bar{X}_{T}^{(u)}(s_{2}) \\
&\quad + \left\{{1 \over Th}\sum_{j=t+1}^{T}K_{1,h}\left(u-{j \over T}\right)(X_{j,T}(s_{1})-X_{j}^{(u)}(s_{1}))\right\}\bar{X}_{T}^{(u)}(s_{2})\\
&\quad =: \gamma_{t,21}(s_{1},s_{2}) + \gamma_{t,22}(s_{1},s_{2}). 
\end{align*}
For $\gamma_{t,21}(s_{1},s_{2})$, 
\begin{align}\label{ga-t21-1}
\int \int (\gamma_{t,21}(s_{1},s_{2}))^{2}ds_{1}ds_{2} &= \left(\int \left({1 \over Th}\sum_{j=t+1}^{T}K_{1,h}\left(u-{j \over T}\right)X_{j}^{(u)}(s_{1})\right)^{2}ds_{1}\right)\left(\int (\bar{X}_{T}^{(u)}(s_{2}))^{2}ds_{2}\right).
\end{align}
Since 
\begin{align*}
&E\left[\int \left({1 \over Th}\sum_{j=t+1}^{T}K_{1,h}\left(u-{j \over T}\right)X_{j}^{(u)}(s)\right)^{2}ds\right]\\
&\quad = {1 \over (Th)^{2}}\sum_{j=t+1}^{T}K_{1,h}^{2}\left(u - {j \over T}\right)E\left[\int (X_{j}^{(u)}(s))^{2}ds\right]\\
&\quad \quad + {2 \over (Th)^{2}}\sum_{t+1\leq j_{1}<t_{2}\leq T}K_{1,h}\left(u - {j_{1} \over T}\right)K_{1,h}\left(u - {j_{2} \over T}\right)E\left[\int X_{j_{1}}^{(u)}(s) X_{j_{2}}^{(u)}(s)ds\right]\\
&\quad  = {v_{2}^{2}(X_{0}^{(u)}) \over (Th)^{2}}\sum_{j=t+1}^{T}K_{1,h}^{2}\left(u - {j \over T}\right)\\
&\quad \quad + {2 \over (Th)^{2}}\sum_{t+1\leq j_{1}<t_{2}\leq T}K_{1,h}\left(u - {j_{1} \over T}\right)K_{1,h}\left(u - {j_{2} \over T}\right)E\left[\int X_{j_{1}}^{(u)}(s) (X_{j_{2}}^{(u)}(s) - X_{j_{2}-j_{1},j_{2}}^{(u)}(s))ds\right]\\
&\quad \lesssim {1 \over (Th)^{2}}\sum_{j=t+1}^{T}K_{1,h}^{2}\left(u - {j \over T}\right)\\
&\quad \quad + {1 \over (Th)^{2}}\sum_{j=1}^{T}K_{1,h}\left(u - {j \over T}\right)\sum_{k \geq 1}v_{2}(X_{0}^{(u)})v_{2}(X_{0}^{(u)} - X_{k,0}^{(u)}) = O\left({1 \over Th}\right), 
\end{align*}
we then have 
\begin{align}\label{ga-t21-2}
\max_{1 \leq t \leq T}E\left[\int \left({1 \over Th}\sum_{j=t+1}^{T}K_{1,h}\left(u-{j \over T}\right)X_{j}^{(u)}(s)\right)^{2}ds\right] = O\left({1 \over Th}\right).
\end{align}
Combining (\ref{ga-t21-1}) and (\ref{ga-t21-2}), we have  
\begin{align}\label{ga-t21-bound}
&E\left[\left(\int \int \left(\sum_{t=1}^{T-1}K_{2}(t/b)\gamma_{t,21}(s_{1},s_{2})\right)^{2}ds_{1}ds_{2}\right)^{1/2}\right] \nonumber \\
&\quad \leq \sum_{t=1}^{T-1}|K_{2}(t/b)|\max_{1 \leq t \leq T}E\left[\int \left({1 \over Th}\sum_{j=t+1}^{T}K_{1,h}\left(u-{j \over T}\right)X_{j}^{(u)}(s)\right)^{2}ds\right]^{1/2}\!\!\!\!E[\|\bar{X}_{T}^{(u)}\|^{2}]^{1/2} \lesssim {b \over Th} \to 0. 
\end{align}
Likewise, we can show
\begin{align}\label{ga-t22-bound}
E\left[\left(\int \int \left(\sum_{t=1}^{T-1}K_{2}(t/b)\gamma_{t,22}(s_{1},s_{2})\right)^{2}ds_{1}ds_{2}\right)^{1/2}\right]  &\lesssim \sqrt{b^{2}h \over T} \to 0. 
\end{align}

(Step3-3) For $\gamma_{t,3}(s_{1},s_{2})$, decompose
\begin{align*}
\gamma_{t,3}(s_{1},s_{2}) &= \left\{{1 \over Th}\sum_{j=t+1}^{T}K_{1,h}\left(u-{j \over T}\right)X_{j-t}^{(u)}(s_{2})\right\}\bar{X}_{T}^{(u)}(s_{1})\\
&\quad + \left\{{1 \over Th}\sum_{j=t+1}^{T}K_{1,h}\left(u-{j \over T}\right)(X_{j-t,T}(s_{2})-X_{j-t}^{(u)}(s_{2}))\right\}\bar{X}_{T}^{(u)}(s_{1})\\
& =: \gamma_{t,31}(s_{1},s_{2}) + \gamma_{t,32}(s_{1},s_{2}).
\end{align*}
For $\gamma_{t,32}(s_{1},s_{2})$, applying almost the same argument to show (\ref{ga-t12}), we have
\begin{align}\label{ga-t32-bound}
&E\left[\left(\int \int \left(\sum_{t=1}^{T-1}K_{2}(t/b)\gamma_{t,32}(s_{1},s_{2})\right)^{2}ds_{1}ds_{2}\right)^{1/2}\right] \nonumber \\
&\quad \leq \sum_{t=1}^{T-1}|K_{2}(t/b)|E\left[\int \left({1 \over Th}\sum_{j=t+1}^{T}K_{1,h}\left(u-{j \over T}\right)(X_{j-t,T}(s)-X_{j-t}^{(u)}(s))\right)^{2}ds\right]^{1/2}\!\!\!\!E[\|\bar{X}_{T}^{(u)}\|^{2}]^{1/2} \nonumber \\
&\quad \lesssim \sum_{t=1}^{b}{t+1 \over T}\sqrt{1 \over Th} \sim {b^{2} \over T\sqrt{Th}} \lesssim {T^{2}h^{2} \over T\sqrt{Th}} = \sqrt{Th^{3}}  \to 0. 
\end{align}
Likewise, applying almost the same argument to show (\ref{ga-t21-bound}), we have 
\begin{align}\label{ga-t31-bound}
&E\left[\left(\int \int \left(\sum_{t=1}^{T-1}K_{2}(t/b)\gamma_{t,31}(s_{1},s_{2})\right)^{2}ds_{1}ds_{2}\right)^{1/2}\right] \nonumber \\
&\quad \leq \sum_{t=1}^{T-1}|K_{2}(t/b)|\max_{1 \leq t \leq T}E\left[\int \left({1 \over Th}\sum_{j=t+1}^{T}K_{1,h}\left(u-{j \over T}\right)X_{j-t}^{(u)}(s)\right)^{2}ds\right]^{1/2}\!\!\!\!E[\|\bar{X}_{T}^{(u)}\|^{2}]^{1/2} \lesssim {b \over Th} \to 0. 
\end{align}

(Step3-4) For $\gamma_{t,4}(s_{1},s_{2})$, we have 
\begin{align*}
\left(\int \int (\gamma_{t,4}(s_{1},s_{2}))^{2}ds_{1}ds_{2}\right)^{1/2} &= {1 \over Th}\sum_{j=t+1}^{T}K_{1,h}\left(u-{j \over T}\right)\|\bar{X}_{T}^{(u)}\|^{4}\\
&\leq {1 \over Th}\sum_{t=1}^{T}K_{1,h}\left(u-{t \over T}\right)\|\bar{X}_{T}^{(u)}\|^{2} \lesssim \|\bar{X}_{T}^{(u)}\|^{2}.
\end{align*}
Then we have 
\begin{align}\label{ga-t4-bound}
E\left[\left(\int \int \left(\sum_{t=1}^{T-1}K_{2}(t/b)\gamma_{t,4}(s_{1},s_{2})\right)^{2}ds_{1}ds_{2}\right)^{1/2}\right] &\lesssim \sum_{t=1}^{T-1}|K_{2}(t/b)|E[\|\bar{X}_{T}^{(u)}\|^{2}] \lesssim {b \over Th} \to 0. 
\end{align}

(Step3-5) Combining (\ref{ga-t12})-(\ref{ga-t14}), (\ref{ga-t21-bound}), (\ref{ga-t22-bound}), (\ref{ga-t32-bound}), (\ref{ga-t31-bound}), and (\ref{ga-t4-bound}), we have (\ref{cov-kernel-conv2}).

(Step4) In this step, we will show 
\begin{align}\label{cov-kernel-conv3}
\int \int \left(\sum_{t=1}^{T-1}K_{2}(t/b)\gamma_{t,11}(s_{1},s_{2}) - c_{1}^{(u)}(s_{1},s_{2})\right)^{2}ds_{1}ds_{2} = o_{P}(1). 
\end{align}
Define
\begin{align*}
\gamma_{t,11}^{(m)}(s_{1},s_{2}) &= {1 \over Th}\sum_{j=t+1}^{T}K_{1,h}\left(u-{j \over T}\right)X_{m,j}^{(u)}(s_{1})X_{m,j-t}^{(u)}(s_{2}),\\
c_{1}^{(u)}(s_{1},s_{2}) &= \sum_{t=1}^{\infty}E[X_{0}^{(u)}(s_{2})X_{t}^{(u)}(s_{1})],\ c_{1,m}^{(u)}(s_{1},s_{2}) = \sum_{t=1}^{m}E[X_{m,0}^{(u)}(s_{2})X_{m,t}^{(u)}(s_{1})],\\
\bar{c}_{1,m}^{(u)}(s_{1},s_{2}) &= \sum_{t=1}^{m}\omega_{t,T}^{(u)}E[X_{m,0}^{(u)}(s_{2})X_{m,t}^{(u)}(s_{1})],
\end{align*}
where $\omega_{t,T}^{(u)} = {1 \over Th}\sum_{j=t+1}^{T}K_{1,h}(u-j/T)$. Note that $\max_{1\leq t \leq T-1}\omega_{t,T}^{(u)} \lesssim 1$ uniformly in $u$. 
Let $m \geq 1$ be a fixed constant. Observe that
\begin{align*}
&\int \!\! \int \!\! \left(\sum_{t=1}^{T-1}K_{2}(t/b)\gamma_{t,11}(s_{1},s_{2}) - c_{1}^{(u)}(s_{1},s_{2})\right)^{2}ds_{1}ds_{2}\\
&\quad \leq 64\!\!  \int \!\! \int \!\! \left(\sum_{t=1}^{T-1}K_{2}(t/b)(\gamma_{t,11}(s_{1},s_{2}) - \gamma_{t,11}^{(m)}(s_{1},s_{2}))\right)^{2}ds_{1}ds_{2}\\
&\quad \quad + 64\!\!  \int \!\! \int \!\! \left(\sum_{t=m+1}^{T-1}K_{2}(t/b)\gamma_{t,11}^{(m)}(s_{1},s_{2}) \right)^{2}\!\! ds_{1}ds_{2} + 64\!\! \int \!\! \int \!\! \left(\sum_{t=1}^{m}(K_{2}(t/b)-1)\gamma_{t,11}^{(m)}(s_{1},s_{2})\right)^{2}\!\!ds_{1}ds_{2}\\
&\quad \quad + 64\!\! \int \!\! \int \!\! \left(\sum_{t=1}^{m}\gamma_{t,11}^{(m)}(s_{1},s_{2}) - \bar{c}_{1,m}^{(u)}(s_{1},s_{2})\right)^{2}\!\! ds_{1}ds_{2}+ 64\!\!  \int \!\! \int \!\! \left(\bar{c}_{1,m}^{(u)}(s_{1},s_{2}) - c_{1,m}^{(u)}(s_{1},s_{2})\right)^{2}\!\! ds_{1}ds_{2}\\ 
&\quad \quad+ 64\int \int \left(c_{1,m}^{(u)}(s_{1},s_{2}) - c_{1}^{(u)}(s_{1},s_{2})\right)^{2}ds_{1}ds_{2}\\
&\quad =: 64(q_{1,T} + q_{2,T} + q_{3,T} + q_{4,T} + q_{5,T}+q_{6,T}). 
\end{align*}

(Step4-1) For $q_{6,T}$, since $\int \int (c_{m}^{(u)}(s_{1},s_{2}) - c^{(u)}(s_{1},s_{2}))^{2}ds_{1}ds_{2}  \to 0$ as $m \to \infty$, we have 
\begin{align}\label{q6-bound}
q_{5,T} = \int \int \left(c_{1,m}^{(u)}(s_{1},s_{2}) - c_{1}^{(u)}(s_{1},s_{2})\right)^{2}ds_{1}ds_{2} \to 0\ \text{as $m \to \infty$}.
\end{align} 

(Step4-2) For $q_{5,T}$, 
\begin{align}\label{q5-bound}
q_{5,T}^{1/2} &\leq \sum_{t=1}^{m}|\omega_{t,T}^{(u)}-1|\left(\int \int (c_{1,m}^{(u)}(s_{1},s_{2}))^{2}ds_{1}ds_{2}\right)^{1/2} \to 0\ \text{as $T \to \infty$}
\end{align}
since $\lim_{T \to \infty}|\omega_{t,T}^{(u)}-1|$ for each $1 \leq t \leq m$. 

(Step4-3) In this step, we will show that for every $m \geq 1$, $q_{4,T} = o_{P}(1)$ as $T \to \infty$.
Applying the ergodic theorem for weighted random variables in a Hilbert space, for $1 \leq t \leq m$, we have
\begin{align*}
(\omega_{t,T}^{(u)})^{2}\int \int \left((\omega_{t,T}^{(u)})^{-1}\gamma_{t,11}^{(m)}(s_{1},s_{2}) - E[X_{m,0}^{(u)}(s_{1})X_{m,t}^{(u)}(s_{2})]\right)^{2}ds_{1}ds_{2} = o_{P}(1).
\end{align*}
Then using the triangular inequality, we have
\begin{align}\label{q4-bound}
q_{4,T}^{1/2} &\leq \sum_{t=1}^{m}\omega_{t,T}^{(u)}\left(\int \int \left((\omega_{t,T}^{(u)})^{-1}\gamma_{t,11}^{(m)}(s_{1},s_{2}) - E[X_{m,0}^{(u)}(s_{1})X_{m,t}^{(u)}(s_{2})]\right)^{2}ds_{1}ds_{2}\right)^{1/2} = o_{P}(1).
\end{align}

(Step4-4) For $q_{3,T}$, applying the triangular inequality, we have 
\begin{align*}
q_{3,T}^{1/2} &\leq \sum_{t=1}^{m}|K_{2}(t/b)-1| \left(\int \int \left(\gamma_{t,11}^{(m)}(s_{1},s_{2})\right)^{2}ds_{1}ds_{2}\right)^{1/2}\\
&\leq \sum_{t=1}^{m}|K_{2}(t/b)-1| \left|\omega_{t,T}^{(u)}\right|  \left(\int \int \left(E[X_{m,0}^{(u)}(s_{2})X_{m,t}^{(u)}(s_{1})]\right)^{2}ds_{1}ds_{2}\right)^{1/2}\\
&\quad + \sum_{t=1}^{m}|K_{2}(t/b)-1| \left|\omega_{t,T}^{(u)}\right| \left(\int \int \left((\omega_{t,T}^{(u)})^{-1}\gamma_{t,11}^{(m)}(s_{1},s_{2}) - E[X_{m,0}^{(u)}(s_{2})X_{m,t}^{(u)}(s_{1})]\right)^{2}ds_{1}ds_{2}\right)^{1/2}.
\end{align*}
Assumption \ref{Ass-KB2} yields that $\max_{1 \leq t \leq m}|K_{2}(t/b)-1| \to 0$ as $T \to \infty$. Then combining the results in Steps 4-2 and 4-3, we have
\begin{align}\label{q3-bound}
q_{3,T} &= o_{P}(1)\ \text{as $T \to \infty$}. 
\end{align}

(Step4-5) Let $\lfloor x \rfloor$ denote the integer part of $x \in \mathbb{R}$. For $q_{2,T}$, observe that
\begin{align*}
&E\left[q_{2,T}\right]\\
&\quad = \sum_{t_{1}=m+1}^{\lfloor C_{2}b \rfloor + 1}\sum_{t_{2}=m+1}^{\lfloor C_{2}b \rfloor + 1}K_{2}(t_{1}/b)K_{2}(t_{2}/b){1 \over (Th)^{2}}\sum_{j_{1}=t_{1}+1}^{T}\sum_{j_{2}=t_{2}+1}^{T}K_{1,h}\left(u-{j_{1} \over T}\right)K_{1,h}\left(u-{j_{2} \over T}\right)\\
&\quad \quad \times \int \int E\left[X_{m,j_{1}}^{(u)}(s_{1})X_{m,j_{1}-t_{1}}^{(u)}(s_{1})X_{m,j_{2}}^{(u)}(s_{2})X_{m,j_{2}-t_{2}}^{(u)}(s_{2})\right]ds_{1}ds_{2}.
\end{align*}
Note that $\{X_{m,t}^{(u)}\}_{t \in \mathbb{Z}}$ is an $m$-dependent sequence and $t_{1}, t_{2} \geq m+1$. Then $X_{m,j_{1}}^{(u)}$ and $X_{m,j_{1}-t_{1}}^{(u)}$ ($X_{m,j_{2}}^{(u)}$ and $X_{m,j_{2}-t_{2}}^{(u)}$) are independent. This implies that the number of terms when the term $E\left[X_{m,j_{1}}^{(u)}(s_{1})X_{m,j_{1}-t_{1}}^{(u)}(s_{1})X_{m,j_{2}}^{(u)}(s_{2})X_{m,j_{2}-t_{2}}^{(u)}(s_{2})\right]$ is not zero is $O(bTh)$. Consequently, 
\begin{align}\label{q2-bound}
E\left[q_{2,T}\right] &\lesssim {b \over Th} \to 0\ \text{as $T \to \infty$}. 
\end{align}

(Step4-6) In this step, we will show that for any $r>0$,
\begin{align}\label{q1-bound}
\lim_{m \to \infty}\lim_{T \to \infty}P\left(q_{1,T}>r\right)=0. 
\end{align}
Observe that
\begin{align*}
&\sum_{t=1}^{T-1}K_{2}(t/b)(\gamma_{t,11}(s_{1},s_{2}) - \gamma_{t,11}^{(m)}(s_{1},s_{2}))\\
&\quad = {1 \over Th}\sum_{t=1}^{T-1}K_{2}(t/b)\sum_{j=t+1}^{T}K_{1,h}\left(u-{j \over T}\right)(X_{j}^{(u)}(s_{1})X_{j-t}^{(u)}(s_{2})-X_{m,j}^{(u)}(s_{1})X_{m,j-t}^{(u)}(s_{2}))\\
&\quad = {1 \over Th}\left(\sum_{t=1}^{m} + \sum_{t=m+1}^{\lfloor C_{2}b \rfloor +1}\right)K_{2}(t/b)\sum_{j=t+1}^{T}K_{1,h}\left(u-{j \over T}\right)(X_{j}^{(u)}(s_{1})X_{j-t}^{(u)}(s_{2})-X_{m,j}^{(u)}(s_{1})X_{m,j-t}^{(u)}(s_{2}))\\
&\quad =: q_{11,T}^{(m)}(s_{1},s_{2}) + q_{12,T}^{(m)}(s_{1},s_{2}). 
\end{align*}
Decompose
\begin{align*}
&X_{j}^{(u)}(s_{1})X_{j-t}^{(u)}(s_{2})-X_{m,j}^{(u)}(s_{1})X_{m,j-t}^{(u)}(s_{2})\\
&\quad = (X_{j}^{(u)}(s_{1}) - X_{m,j}^{(u)}(s_{1}))X_{j-t}^{(u)}(s_{2}) +  (X_{j-t}^{(u)}(s_{2}) - X_{m,j-t}^{(u)}(s_{2}))X_{m,j}^{(u)}(s_{1}). 
\end{align*}
Then applying the triangular inequality, we have 
\begin{align}\label{q_11t-1}
&E\left[\left\{\int \int \left({1 \over Th}\sum_{t=1}^{m}K_{2}(t/b)\sum_{j=t+1}^{T}K_{1,h}\left(u-{j \over T}\right)(X_{j}^{(u)}(s_{1}) - X_{m,j}^{(u)}(s_{1}))X_{j-t}^{(u)}(s_{2})\right)^{2}ds_{1}ds_{2}\right\}^{1/2}\right] \nonumber \\
&\quad \leq  {1 \over Th}\sum_{t=1}^{m}|K_{2}(t/b)|\sum_{j=t+1}^{T}K_{1,h}\left(u-{j \over T}\right) \nonumber \\
&\quad \quad \times E\left[\left(\int (X_{j}^{(u)}(s_{1}) - X_{m,j}^{(u)}(s_{1}))^{2}ds_{1}\right)^{1/2}\left(\int (X_{j-t}^{(u)}(s_{2}))^{2}ds_{2}\right)^{1/2}\right] \nonumber \\
&\quad \lesssim \left({1 \over Th}\sum_{t=1}^{T}K_{1,h}\left(u-{t \over T}\right)\right)mE\left[\|X_{0}^{(u)} - X_{m,0}^{(u)}\|^{2}\right]^{1/2}E\left[\|X_{0}^{(u)}\|^{2}\right]^{1/2} \nonumber \\
&\quad \lesssim mv_{2}(X_{0}^{(u)} - X_{m,0}^{(u)}) \to 0\ \text{as $m \to \infty$}.
\end{align}
Likewise, we can show 
\begin{align}\label{q_11t-2}
&E\left[\left\{\int \int \left({1 \over Th}\sum_{t=1}^{m}K_{2}(t/b)\sum_{j=t+1}^{T}K_{1,h}\left(u-{j \over T}\right)(X_{j-t}^{(u)}(s_{2}) - X_{m,j-t}^{(u)}(s_{2}))X_{m,j}^{(u)}(s_{1})\right)^{2}ds_{1}ds_{2}\right\}^{1/2}\right] \nonumber \\
&\quad \lesssim mv_{2}(X_{0}^{(u)} - X_{m,0}^{(u)}) \to 0\ \text{as $m \to \infty$}.
\end{align}
Combining (\ref{q_11t-1}) and (\ref{q_11t-2}), we have
\begin{align}\label{q_11t-conv1}
\lim_{m \to \infty}\lim_{T \to \infty}P\left(\int \int (q_{11,T}^{(m)}(s_{1},s_{2}))^{2}ds_{1}ds_{2}>r_{1}\right) = 0 
\end{align}
for any $r_{1}>0$.

Further, decompose
\begin{align*}
X_{j}^{(u)}(s_{1})X_{j-t}^{(u)}(s_{2}) &= (X_{j}^{(u)}(s_{1}) - X_{t,j}^{(u)}(s_{1}))X_{j-t}^{(u)}(s_{2}) + X_{t,j}^{(u)}(s_{1})X_{j-t}^{(u)}(s_{2}). 
\end{align*}
Then we have 
\begin{align}\label{q_12t-1}
&E\left[\left\{\int \int \left({1 \over Th}\sum_{t=m+1}^{\lfloor C_{2}b \rfloor +1}K_{2}(t/b)\sum_{j=t+1}^{T}K_{1,h}\left(u-{j \over T}\right)(X_{j}^{(u)}(s_{1}) - X_{t,j}^{(u)}(s_{1}))X_{j-t}^{(u)}(s_{2})\right)^{2}ds_{1}ds_{2}\right\}^{1/2}\right] \nonumber \\
&\quad \leq  {1 \over Th}\sum_{t=m+1}^{\lfloor C_{2}b \rfloor +1}|K_{2}(t/b)|\sum_{j=t+1}^{T}K_{1,h}\left(u-{j \over T}\right) \nonumber \\
&\quad \quad \times E\left[\left(\int (X_{j}^{(u)}(s_{1}) - X_{t,j}^{(u)}(s_{1}))^{2}ds_{1}\right)^{1/2}\left(\int (X_{j-t}^{(u)}(s_{2}))^{2}ds_{2}\right)^{1/2}\right] \nonumber \\
&\quad \lesssim \left({1 \over Th}\sum_{t=1}^{T}K_{1,h}\left(u-{t \over T}\right)\right)\sum_{t=m+1}^{\lfloor C_{2}b \rfloor +1}E\left[\|X_{0}^{(u)} - X_{t,0}^{(u)}\|^{2}\right]^{1/2}E\left[\|X_{0}^{(u)}\|^{2}\right]^{1/2} \nonumber \\
&\quad \lesssim \sum_{t=m+1}^{\infty}v_{2}(X_{0}^{(u)} - X_{m,0}^{(u)}) \to 0\ \text{as $m \to \infty$}.
\end{align}
Applying almost the same argument to show (\ref{q2-bound}), we have
\begin{align}\label{q_12t-2}
&E\left[\left\{\int \int \left({1 \over Th}\sum_{t=m+1}^{\lfloor C_{2}b \rfloor +1}K_{2}(t/b)\sum_{j=t+1}^{T}K_{1,h}\left(u-{j \over T}\right)X_{t,j}^{(u)}(s_{1})X_{j-t}^{(u)}(s_{2})\right)^{2}ds_{1}ds_{2}\right\}^{1/2}\right] \nonumber \\
&\quad \lesssim {b \over Th} \to 0\ \text{as $T \to \infty$}. 
\end{align}
Combining (\ref{q_12t-1}) and (\ref{q_12t-2}), we have
\begin{align}\label{q_12t-a}
\lim_{m \to \infty}\lim_{T \to \infty}P\left(\int \int \left({1 \over Th}\sum_{t=m+1}^{T-1}K_{2}(t/b)\sum_{j=t+1}^{T}K_{1,h}\left(u-{j \over T}\right)X_{j}^{(u)}(s_{1})X_{j-t}^{(u)}(s_{2})\right)^{2}ds_{1}ds_{2}>r_{2}\right) = 0
\end{align}
for any $r_{2}>0$. Likewise, we can show 
\begin{align}\label{q_12t-b}
\lim_{m \to \infty}\lim_{T \to \infty}\!\!P\!\left(\int \!\! \int \!\! \left({1 \over Th}\sum_{t=m+1}^{T-1}K_{2}(t/b)\sum_{j=t+1}^{T}K_{1,h}\left(u-{j \over T}\right)X_{m,j}^{(u)}(s_{1})X_{m,j-t}^{(u)}(s_{2})\right)^{2}ds_{1}ds_{2}>r_{2}\right) = 0
\end{align}
for any $r_{2}>0$. 
Combining (\ref{q_12t-a}) and (\ref{q_12t-a}), we have
\begin{align}\label{q_12t-conv1}
\lim_{m \to \infty}\lim_{T \to \infty}P\left(\int \int (q_{12,T}^{(m)}(s_{1},s_{2}))^{2}ds_{1}ds_{2}>r_{1}\right) = 0
\end{align}
for any $r_{1}>0$. Consequently, (\ref{q_11t-conv1}) and (\ref{q_12t-conv1}) yield (\ref{q1-bound}). 

(Step4-7) Combining (\ref{q6-bound}), (\ref{q5-bound}), (\ref{q4-bound}), (\ref{q3-bound}), (\ref{q2-bound}), and (\ref{q1-bound}) yield (\ref{cov-kernel-conv3}). 
\end{proof}

\subsection{Proofs for Section \ref{Sec: Applications}}

We omit the proofs of Propositions \ref{Prop1}, \ref{Prop2} and \ref{Prop3} since they immediately follow from the results in Section \ref{Sec: Main}. 

\begin{proof}[Proof of Corollary \ref{Cor2}]
For $1 \leq j \leq q_{0}-1$, from (\ref{two-sample-eigen-consistency}) with $q = q_{0}$, we have 
\begin{align}\label{Cor2-1}
{\hat{\eta}_{u,j+1} \over \hat{\eta}_{u,j}} \stackrel{p}{\to} {\eta_{u,j+1} \over \eta_{u,j}}>0.
\end{align}
For $j \geq q_{0}$, applying Lemma \ref{lem: HK-lem22}, we have
\begin{align}\label{Cor2-1-1}
|\hat{\eta}_{u,j+1}|  \leq \|\hat{\Phi}^{(u)} - \Phi^{(u)}\|_{\mathcal{S}} &\stackrel{p}{\to} 0.
\end{align} 
Thus we have
\begin{align}\label{Cor2-2}
{\hat{\eta}_{u,q_{0}+1} \over \hat{\eta}_{u,q_{0}}} &\stackrel{p}{\to} 0.
\end{align}
For $j > q_{0}$, (\ref{Cor2-1-1}) yields
\begin{align*}
P\left(\left|{\hat{\eta}_{u,j} \over \hat{\eta}_{u,1}}\right| < \varepsilon_{0}, \left|{\hat{\eta}_{u,j+1} \over \hat{\eta}_{u,1}}\right| < \varepsilon_{0}\right) &\to 1. 
\end{align*}
Hence 
\begin{align}\label{Cor2-3}
P\left({\hat{\eta}_{u,j+1} \over \hat{\eta}_{u,j}} = 0/0 =1\right) &\to 1. 
\end{align}
With (\ref{Cor2-1}), (\ref{Cor2-2}) and (\ref{Cor2-3}), we complete the proof.
\end{proof}

\section{Auxiliary Lemmas}

In this section, we provide some auxiliary lemmas used in the proofs of main results. Recall that $\mathcal{L}$ denotes the space of bounded (continuous) linear operators on $H$ with the norm 
\[
\|\Psi\|_{\mathcal{L}} = \sup\{\|\Psi(x)\|: \|x\|\leq 1\}. 
\] 
Let $C_{1}, C_{2} \in \mathcal{L}$ be two compact operators with the following singular value decomposition:
\begin{align}\label{def: compact-op}
C_{1}(x) &= \sum_{j \geq 1}\lambda_{1,j}\langle x, v_{1,j} \rangle f_{1,j},\ C_{2}(x) = \sum_{j \geq 1}\lambda_{2,j}\langle x, v_{2,j} \rangle f_{2,j}\ x \in H,
\end{align}
where $\{\lambda_{1,j}\}$ and $\{\lambda_{2,j}\}$ are sequences of nonnegative constants, and $\{v_{1,j}\}$, $\{v_{2,j}\}$, $\{f_{1,j}\}$ and $\{f_{2,j}\}$ are orthonormal bases of $H$.

\begin{lemma}[Lemma 1.6 in \cite{GoGoKa90}]\label{lem: HK-lem22}
Suppose $C_{1},C_{2} \in \mathcal{L}$ are two compact operators with singular value decompositions (\ref{def: compact-op}). Then, for each $j \geq 1$, $|\lambda_{1,j} - \lambda_{2,j}| \leq \|C_{1}-C_{2}\|_{\mathcal{L}}$. 
\end{lemma}

Now we assume that a compact operator $C_{2}$ with the singular value decomposition (\ref{def: compact-op}) is symmetric and $C_{2}(v_{2,j})=\lambda_{2,j}v_{2,j}$, that is, $f_{2,j}=v_{2,j}$. Note that any covariance operator $C$ satisfies these conditions. Define
\[
v'_{2,j} = c_{2,j}v_{2,j},\ c_{2,j} = \text{sign}(\langle v_{1,j}, v_{2,j}\rangle).
\]

\begin{lemma}[Lemma 2.3 in \cite{HoKo12}]\label{lem: HK-lem23} 
Suppose $C_{1},C_{2} \in \mathcal{L}$ are two compact operators with singular value decompositions (\ref{def: compact-op}). If $C_{2}$ is symmetric and $f_{2,j} = v_{2,j}$ in (\ref{def: compact-op}), and its eigenvalues satisfy $\lambda_{2,1}>\lambda_{2,2}>\dots>\lambda_{2,q}>\lambda_{2,q+1}$, then
\[
\|v_{1,j} - v'_{2,j}\| \leq {2\sqrt{2} \over \alpha_{2,j}}\|C_{1}-C_{2}\|_{\mathcal{L}},\ 1\leq j \leq q,
\]
where $\alpha_{2,1} = \lambda_{2,1}-\lambda_{2,2}$ and $\alpha_{2,j}=\min(\lambda_{2,j-1}-\lambda_{2,j}, \lambda_{2,j}-\lambda_{2,j+1})$, $2\leq j\leq q$. 
\end{lemma}

The proof of following Lemma \ref{lem: g} is straightforward and thus omitted. 
\begin{lemma}\label{lem: g}
Suppose that kernel $K_{1}$ satisfies Assumption \ref{Ass-KB}. Then, 
\begin{align*}
\sup_{u \in [C_{1}h, 1-C_{1}h]}\left|{1 \over Th}\sum_{t=1}^{T}K_{1,h}\left(u-{t \over T}\right) - 1\right| = O\left({1 \over Th^{2}}\right) + o(h). 
\end{align*}
\end{lemma}

%\section{Note}
%Since
%\begin{align*}
%\left|\|\bar{X}_{T,1}^{(u)} - \bar{X}_{T,2}^{(u)}\| - \|m_{u,1} - m_{u,2}\|\right| &\leq \|\bar{X}_{T,1}-m_{u,1}\| + \|\bar{X}_{T,2}^{(u)} - m_{u,2}\|,
%\end{align*}
%we have $\|\bar{X}_{T,1}^{(u)} - \bar{X}_{T,2}^{(u)}\| - \|m_{u,1} - m_{u,2}\| = O_{P}\left({1 \over \sqrt{T_{1}h}} + {1 \over \sqrt{T_{2}h}}\right)$. 

%\section*{Acknowledgments}
%D. Kurisu is partially supported by JSPS KAKENHI Grant Number 20K13468. The author would like to thank Alessia Caponera, Konstantinos Fokianos, Kengo Kato, Muneya Matsui, and Taisuke Otsu for their helpful comments and suggestions.

\end{document}